\documentclass[11pt,oneside]{amsart}
\usepackage[top=25mm, bottom=25mm, left=20mm, right=20mm]{geometry}
\usepackage{amsfonts,amssymb,amsmath,amsthm,amsxtra}
\usepackage[T1]{fontenc}
\usepackage[utf8]{inputenc}
\usepackage{bm}
\usepackage{mathtools}
\usepackage{dsfont}
\usepackage{mathrsfs}
\usepackage{enumitem}
\allowdisplaybreaks

\usepackage{graphics}
\usepackage{hyperref}

\numberwithin{equation}{section}
\newtheorem{theorem}[equation]{Theorem}

\newtheorem{lemma}[equation]{Lemma}
\newtheorem{proposition}[equation]{Proposition}
\theoremstyle{definition}
\newtheorem{remark}[equation]{Remark}
\newtheorem{definition}[equation]{Definition}

\newtheorem*{claim}{Claim}

\DeclareMathOperator\supp{supp}

\begin{document}
\title[Weak-type (1,1) maximal inequality and pointwise ergodic theorems along thin sets]{Weak-type (1,1) inequality for discrete maximal functions and pointwise ergodic theorems along thin arithmetic sets}
\author{Leonidas Daskalakis}
\thanks{Department of Mathematics, Rutgers University, Leonidas Daskalakis is supported by the NSF grant DMS-2154712.}
\maketitle

\begin{abstract}
 We establish weak-type $(1,1)$ bounds for the maximal function associated with ergodic averaging operators modeled  on a wide class of thin deterministic sets $B$. As a corollary we obtain the corresponding pointwise convergence result on $L^1$. This contributes yet another counterexample for the conjecture of Rosenblatt and Wierdl from 1991 asserting the failure of pointwise convergence on $L^1$ of ergodic averages along arithmetic sets with zero Banach density. The second main result is a multiparameter pointwise ergodic theorem in the spirit of Dunford and Zygmund along $B$ on $L^p$, $p>1$, which is derived by establishing uniform oscillation estimates and certain vector-valued maximal estimates.
\end{abstract}

\section{Introduction}
\subsection{Brief historical remarks}
In 1991 Rosenblatt and Wierdl \cite[Conjecture~4.1]{RWC} formulated a famous conjecture asserting that for any arithmetical set $A$ with zero Banach density and for any $(X,\mathcal{B},\mu,T)$ aperiodic probability dynamical system, there exists a function $f\in L^1_{
\mu}(X)$, such that 
\[
M_{A,N}f=\frac{1}{|A\cap[1,N]|}\sum_{n\in A\cap[1,N]}f\circ T^n\qquad \text{ does not converge almost everywhere,}
\]
i.e. the set of $x\in X$ such that $\lim_{N\to\infty}M_{A,N}f(x)$ does not exist has positive measure. This was disproven in 2006 by Buczolich \cite{BUL1}, where he provided a counterexample by constructing inductively an appropriate set $A$ of zero Banach density for which one gets the pointwise convergence of the ergodic averages $M_{A,N}f$ for all $f\in L^1$.

Calder\'on's transference principle suggests that such questions are closely related to the study of weak-type $(1,1)$ estimates for the maximal function corresponding to those averages over the integer shift system, namely for the operator
\[
\mathcal{M}_{A}f(x)=\sup_{N\in\mathbb{N}}\frac{1}{|A\cap[1,N]|}\sum_{n\in A\cap[1,N]}|f(x-n)|
\]
A year later it was shown \cite{ntoc} that for $A=\big\{\lfloor n^c\rfloor:\,n\in\mathbb{N}\big\}$ with $c\in\big(1,\frac{1001}{1000}\big)$, the operator $\mathcal{M}_{A}$ is of weak-type $(1,1)$, and as a corollary the authors proved pointwise convergence on $L^1$ for the corresponding ergodic averages along the set $A$, providing a counterexample of the aforementioned conjecture given by a concrete formula. This class of examples was extended in \cite{W11} where the author established the weak-type (1,1) bounds for $\mathcal{M}_{A}$ and the corresponding pointwise ergodic theorem on $L^1$ for sets of the form $\big\{\lfloor n^{c}\ell(n) \rfloor:\,n\in\mathbb{N}\big\}$, where $c$ is close to $1$ and $\ell$ is a certain kind of slowly varying function, for example any iterate of $\log$, see Definitions~$\ref{def1},\ref{def0}$ below. 

One of the main results of the present work is a natural extension of the result from \cite{W11} and in order to formulate it, we must introduce two families of functions that one may think of as slowly varying and regularly varying functions respectively.
	\begin{definition}\label{def1}
	Fix $x_0\ge1$ and let $\mathcal{L}$ denote the set of all functions $\ell\colon[x_0,\infty)\to[1,\infty)$ such that
	\[
	\ell(x)=\exp\bigg(\int_{x_0}^x\frac{\vartheta(t)}{t}dt\bigg)
	\] 
	where $\vartheta\in\mathcal{C}^2([x_0,\infty))$ is a real-valued function satisfying
\[
\vartheta(x)\to 0\,\,\text{, }x\vartheta'(x)\to 0\,\,\text{, }x^2\vartheta''(x)\to 0\,\,\text{ as }x\to \infty\text{.}
\]
\end{definition} 

\begin{definition}\label{def0}
	Fix $x_0\ge1$ and let $\mathcal{L}_0$ denote the set of all functions $\ell\colon[x_0,\infty)\to[1,\infty)$ such that
	\[
	\ell(x)=\exp\bigg(\int_{x_0}^x\frac{\vartheta(t)}{t}dt\bigg)
	\] 
	where $\vartheta\in\mathcal{C}^2([x_0,\infty))$ is a positive and decreasing function satisfying
\[
\vartheta(x)\to 0\,\,\text{, }\frac{x\vartheta'(x)}{\vartheta(x)}\to 0\,\,\text{, }\frac{x^2\vartheta''(x)}{\vartheta(x)}\to 0\,\,\text{ as }x\to \infty\text{,}
\]
and such that for all $\varepsilon>0$ we have $\vartheta(x)\gtrsim_{\varepsilon}x^{-\varepsilon}$ and $\lim_{x\to\infty}\ell(x)=\infty$.

\end{definition}

We note that $\mathcal{L}_0\subseteq \mathcal{L}$.
\begin{definition}
Fix $x_0\ge 1$, $c\in(1,\infty)$ and let $\mathcal{R}_c$ be the set of all functions $h\colon [x_0,\infty)\to [1,\infty)$ such that $h$ is strictly increasing, convex and of the form $h(x)=x^c\ell(x)$ for some $\ell\in\mathcal{L}$. We define $\mathcal{R}_1$ analogously, but with the extra assumption that $\ell\in\mathcal{L}_0$. 
\end{definition} 

We are now ready to give the definitions of the arithmetic sets we are interested in. Let $c_1,c_2\in[1,2)$ and let us fix $h_1$ and $h_2$ in $\mathcal{R}_{c_1}$ and $\mathcal{R}_{c_2}$ respectively. Let $\varphi_1$ and $\varphi_2$ be the compositional  inverses of $h_1$ and $h_2$ and for convenience, let $\gamma_1=1/c_1$ and $\gamma_2=1/c_2$.  Let us fix a function $\psi\colon [1,\infty)\to(0,1/2]$, $\psi\in\mathcal{C}^2([1,\infty))$ such that
\[
\psi(x)\sim\varphi_2'(x)\text{ , }\,\,\psi'(x)\sim\varphi_2''(x)\text{ , }\,\,\psi''(x)\sim\varphi_2'''(x)\text{ as }x\to \infty
\]
We can now define $B_+=\{\,n\in\mathbb{N}:\{\varphi_1(n)\}<\psi(n)\,\}$ and $B_-=\{\,n\in\mathbb{N}:\{-\varphi_1(n)\}<\psi(n)\,\}$, where $\{x\}=x-\lfloor x \rfloor$.

Those sets have been introduced and studied in \cite{HLMP}, where the authors proved that the Hardy–Littlewood majorant property holds for them, see Theorem 1 and 2 in \cite{HLMP}, as a corollary of a restriction theorem. Recently, the author \cite{RTP} proved an analogous result for $\mathbb{P}\cap B_{\pm}$, see Theorem~1.8 in \cite{RTP}, as well as Roth's theorem in these sets, namely, it was show that any subset of the primes of the form $B_{\pm}$ with positive relative upper density contains infinitely many non-trivial three-term arithmetic progressions. 

Following \cite{RTP} we repeat and extend some comments on the sophisticated nature of the sets $B_{\pm}$, simultaneously we try to convince the reader about the richness of the family of the sets that we consider. More precisely, to motivate the definition of $B_{\pm}$, we note that
\[
n\in B_-\iff \exists  m\in\mathbb{N}\colon\,0\le m-\varphi_1(n)<\psi(n)\iff \exists m\in\mathbb{N}\colon\,\varphi_1(n)\le m < \varphi_1(n)+\psi(n)\iff
\]  
\[
\exists m\in\mathbb{N}\colon\,n\le h_1(m)<h_1(\varphi_1(n)+\psi(n))\iff \exists m\in\mathbb{N}\colon\,h_1(m)\in [n,h_1(\varphi_1(n)+\psi(n)))
\]
For $\gamma\in(0,1)$, $h_1(x)=h_2(x)=x^{1/\gamma}$ and $\psi(x)=\varphi_1(x+1)-\varphi_1(x)$, note that the last condition becomes $m^{1/\gamma}\in [n,n+1)$ or $n=\lfloor m^{1/\gamma}\rfloor$, thus $B_-=\{\,\lfloor m^{1/\gamma} \rfloor\colon m\in\mathbb{N}\,\}$. It is not difficult to see that any set 
\begin{equation}\label{Nh}
\mathbb{N}_{h}\coloneqq\{\,\lfloor h(m) \rfloor\colon m\in\mathbb{N}\,\},\quad h\in\mathcal{R}_c
\end{equation} can also be brought in the form $B_-$ by similar appropriate choices. Thus the family of sets we consider includes the fractional powers with exponent close to $1$ and even the more general sets considered in \cite{W11}.

\subsection{One-parameter ergodic theorem on $L^1$} We are now ready to state one of the main results of our paper. Due to some technical complications, we demand further that $\varphi_1\simeq\varphi_2$, and note that this implies that $c_1=c_2$.
\begin{theorem}[Weak-type (1,1) inequality for $\mathcal{M}_{B_{\pm}}$]\label{WT11}
Assume $c_1\in (1,30/29)$ and $\varphi_1\simeq\varphi_2$. Then the maximal function
		\[
		\mathcal{M}_{B_{\pm}}f(x)=\sup_{N\in \mathbb{N}}\frac{1}{|B_{\pm}\cap [1,N]|}\sum_{n\in B_{\pm}\cap [1,N]}|f(x-n)|
		\]
		is of weak-type (1,1), i.e.:
		\[
		|\{x\in\mathbb{Z}\,:\,|\mathcal{M}_{B_{\pm}}f(x)|>\lambda\}|\le C \lambda^{-1}\|f\|_{\ell^1(\mathbb{Z})}
		\]
		By interpolation, this implies that for all $p\in (1,\infty]$, there exists a constant $C_p>0$ such that
		\[
		\|\mathcal{M}_{B_{\pm}}f\|_{\ell^{p}(\mathbb{Z})}\le C_p \|f\|_{\ell^p(\mathbb{Z})}\text{.}
		\]
	\end{theorem}
We use this, together with 2-oscillation estimates, see Theorem~$\ref{OM}$, to obtain the following pointwise convergence result.

\begin{theorem}[Pointwise ergodic theorem]\label{PTC}
Assume $c_1\in (1,30/29)$, $\varphi_1\simeq\varphi_2$ and let $(X,\mathcal{B},\mu,T)$ be an invertible $\sigma$-finite measure preserving system. For any $p\in [1,\infty)$ and any $f\in L_{\mu}^p(X)$, we have that
\[
M_{B_{\pm},N}f(x)=\frac{1}{|B_{\pm}\cap [1,N]|}\sum_{n\in B_{\pm}\cap [1,N]}f(T^nx)\quad \text{ converges $\mu$-a.e. on $X$.}
\]
\end{theorem}
Before discussing the strategy of our proofs, we would like to further examine the sets $B_+$ and $B_-$. Let us restrict our attention to the sets $B_+$, which we call $B$ from now on, as the results for $B_-$ are of equal difficulty. Note that 
	\[
	n\in B\iff \exists  m\in\mathbb{N}:\,0\le \varphi_1(n)-m<\psi(n)\iff \exists m\in\mathbb{N}:\,m\in(\varphi_1(n)-\psi(n),\varphi_1(n)]
	\]  
Now assume that $n\in B$, and $m\in \mathbb{N}$ is such that $m\in(\varphi_1(n)-\psi(n),\varphi_1(n)]$ and assume that $n_0$ is the smallest integer such that $m\in(\varphi_1(n_0)-\psi(n_0),\varphi_1(n_0)]$. Even in simple examples, we should expect that $B$ will contain a lot of consecutive integers after $n_0$. For example, for any $\varphi_1$ inverse of a function in $\mathcal{R}_c$, let $\varphi_2=C100\varphi_1$, where $C$ is the doubling constant of $\varphi_1'$, namely, $\varphi_1'(x)\le C \varphi'(2x)$,  and let $\psi=\varphi_2'$. Since $\varphi_1$ is increasing and $\psi=\varphi_2'$ is decreasing, we get that $\varphi_1-\psi$ is increasing and thus since $m\notin (\varphi_1(n_0-1)-\psi(n_0-1),\varphi_1(n_0-1)]$, we get that $m>\varphi_1(n_0-1)$. We claim that for all $l\in\{0,\cdots,99\}$, we get that   $m\in (\varphi_1(n_0+l)-\psi(n_0+l),\varphi_1(n_0+l)]$, which implies that $B$ contains $99$ consecutive integers after $n_0$. Clearly, $m\le \varphi_1(n_0+l)$ for all $l\in\mathbb{N}_0$. If we assume for the sake of a contradiction that for some $l\in\{0,\dotsc,99\}$ we have that $m<\varphi_1(n_0+l)-\psi_1(n_0+l)$, then  
	\[
	\varphi_1(n_0-1)<\varphi_1(n_0+l)-\psi(n_0+l)
	\]
	and by the Mean Value Theorem, there exists a $\xi_{n_0,l}\in(n_0-1,n_0+l)$ such that
	\[
	\frac{C100\varphi_1'(n_0+l)}{l+1}<\frac{\varphi_1(n_0+l)-\varphi_1(n_0-1)}{l+1}=\varphi_1'(\xi_{n_0,l})\le \varphi_1'(n_0-1)\le C\varphi_1'(2n_0-2)\le C\varphi_1'(n_0+l)
	\]
	Thus
	\[
	100<l+1\text{ which is a contradiction.}
	\]
This shows that the set $B$ we considered here contains infinitely many full blocks of 100 consecutive integers. Such a set $B$ stands in sharp contrast to the sets of the form $\mathbb{N}_{h}=\{\,\lfloor h(m) \rfloor\colon m\in\mathbb{N}\,\}$, $h\in\mathcal{R}_c$, as the gaps between members of such sets tend to infinity.

In general, the constant $\sup_{x\in[1,\infty)}\frac{\varphi_2'(x)}{\varphi_1'(x)}$ determines an important qualitative aspect of the form of the sets $B$, see Lemma~$\ref{BVBI}$. Loosely speaking, for big intervals of integers where the ratio is bigger than $L$, we expect that $B$ will contain blocks of length at least $L/C$, where $C$ is the doubling constant of $\varphi_1'$. The technical issues that arose when trying to handle the case where $B$ contains arbitrarily long intervals of integers (specifically in the counting Lemma~$\ref{1Aprox}$) forced the author to impose the restriction $\varphi_1'\simeq \varphi_2'$, which is equivalent to $\varphi_1\simeq \varphi_2$. Even in this simpler case, we note that $\frac{\varphi_2'(x)}{\varphi_1'(x)}$ could oscillate and thus $B$ could contain blocks of various oscillating lengths. We are now ready to make the following remarks.

\begin{remark}[Smooth dyadic maximal operator]\label{SDef}To establish Theorem~$\ref{WT11}$, it is convenient to work with a smooth dyadic variant of the maximal operator. More precisely, 
let's fix $\eta\in \mathcal{C}^{\infty}(\mathbb{R})$ such that $0\le\eta(x)\le1$ for all $x\in\mathbb{R}$, $\supp(\eta)\subseteq (1/2,4)$ and $\eta(x)=1$ for all $x\in[1,2]$. We define
\[
\mathcal{M}^{(sd)}_Bf(x)=\sup_{k\in\mathbb{N}_0}\bigg\{\frac{1}{\varphi_2(2^k)}\sum_{n\in B}|f(x-n)|\eta\Big(\frac{n}{2^k}\Big)\bigg\}
\]
and note that
\[
\mathcal{M}_Bf(x)\lesssim\sup_{k\in\mathbb{N}_0}\bigg\{\frac{1}{|B\cap[2^{k},2^{k+1})|}\sum_{n\in B\cap[2^{k},2^{k+1})}|f(x-n)|\bigg\}\lesssim \mathcal{M}^{(sd)}_Bf(x)
\]
The first inequality is straightforward, and for the second one, note that Lemma~$\ref{TrEst}$ implies that
\[
|B\cap [1,N]|\simeq |B\cap [N/2,N)|\simeq \varphi_2(N)\text{ for sufficiently large }N\text{.}
\]
Thus it suffices to establish the weak-type (1,1) bound for the smooth dyadic maximal function, so let us denote $\mathcal{M}^{(sd)}_B$ by $\mathcal{M}$ and let 
\[
K_N(x)=\frac{1}{\varphi_2(N)}\sum_{n\in B}\delta_{n}(x)\eta\Big(\frac{n}{N}\Big)\text{,}
\]
so that $\mathcal{M}f(x)=\sup_{k\in\mathbb{N}_0} \{K_{2^k}*|f|(x)\}$.
\end{remark}

We give a brief overview of the main ideas of the proof of Theorem~$\ref{WT11}$. We use a subtle variation of the Calder\'on-Zygmund decomposition that was introduced by Fefferman \cite{FE}, see also \cite{MC}, in a similar manner to \cite{ntoc} and \cite{W11}. More specifically, after approximating $K_n*\widetilde{K}_n$ by suitable well-behaving functions, see Lemma~$\ref{2Aprox}$, we employ a refined Calder\'on-Zygmund decomposition which allows us to use $\ell^2$-estimates for the ``very bad'' part of the decomposition, see subsection~$\ref{RBADPART}$. The aforementioned approximation is analogous to the one presented in section 5 of \cite{W11} and similar techniques are used here. The novelty lies in the sophisticated nature of the sets $B$ which complicates the situation substancially. For example, bounding $|K_n*\widetilde{K}_n(x)|$ for small values of $x$, see Lemma~$\ref{1Aprox}$, is precisely what forced the author to impose the extra assumption $\varphi_1\simeq \varphi_2$. To carry out the approximation one needs to estimate certain exponential sums and the main tool is Van der Corput's inequality. Some of the necessary exponential sum estimates can be readily found in section 3 of \cite{W11} and suitable extensions are already established by the author in \cite{RTP}. Finally, we formulate an abstract result, see Theorem~$\ref{absThm}$, which is a generalization of Theorem~6.1 of \cite{W11},  adapted to our approximation for $K_n*\widetilde{K}_n$. We give the full proof of Theorem~$\ref{WT11}$ in section $\ref{WT11Inh}$, see section 3 of \cite{ntoc}, and sections 5 and 6 of \cite{W11}.  

Combining Theorem~$\ref{WT11}$ with the trivial estimate $\|\mathcal{M}_B\|_{\ell^{\infty}(\mathbb{Z})\to\ell^{\infty}(\mathbb{Z})}\lesssim 1$, we obtain by interpolation that $\|\mathcal{M}_{B}\|_{\ell^{p}(\mathbb{Z})\to\ell^{p}(\mathbb{Z})}\lesssim_p 1$, for all $p\in(1,\infty]$. Calder\'on's transference principle implies that for any invertible $\sigma$-finite measure preserving system $(X,\mathcal{B},\mu,T)$ we have 
\[
\big\|\sup_{N\in\mathbb{N}}M_{B,N}f\big\|_{L^p_{\mu}(X)}\lesssim_p\|f\|_{L^p_{\mu}(X)}\text{ for }p\in(1,\infty)\text{, and }\mu(\big\{x\in X:\,|\sup_{N\in\mathbb{N}}M_{B,N}f(x)|>\lambda\big\})\lesssim \frac{\|f\|_{L^{1}_{\mu}(X)}}{\lambda}
\]
A standard argument shows that $\mathcal{L}^p_{B}=\big\{f\in L^p_{\mu}(X):\,\lim_{N\to\infty}M_{B,N}f\text{ exists }\mu\text{--a.e.}\big\}$ is closed in $L_{\mu}^p(X)$, for all $p\in[1,\infty)$, and thus to establish Theorem~$\ref{PTC}$ it suffices to exhibit an $L_{\mu}^p$-dense class of functions $\mathcal{D}_p$ contained in $\mathcal{L}^p_{B}$. The exponential sum estimates of Lemma~$\ref{TrEst}$ together with a straightforward adaptation of the argument presented in section 3 of \cite{W11}, which uses ideas from \cite{BET}, shows that one may take $\mathcal{D}_p=L^p_{\mu}(X)\cap L^2_{\mu}(X)$ and conclude. Instead of this, one can derive the pointwise convergence theorem immediately by the much stronger uniform 2-oscillation $L^p_{\mu}$-estimates of Theorem~$\ref{OM}$, which are also exploited in the sequel.

\subsection{Multi-parameter ergodic theorem} The second main result of our paper is a multi-parameter variant of Theorem~$\ref{PTC}$. Here we discard the assumption $\varphi_1\simeq\varphi_2$ and the acceptable range of $c_1$ and $c_2$ is considerably larger. In contrast to the one-parameter situation, weak-type $(1,1)$ estimates do not hold here.

To make the exposition slightly cleaner, let's fix $k\in\mathbb{N}$ and $B_1,\dotsc,B_k$ as in the introduction where if $(h_{1,i},h_{2,i})\in\mathcal{R}_{c_{1,i}}\times \mathcal{R}_{c_{2,i}} $ are as in the definition of $B_i$, then assume that $c_{1,i}\in[1,2)$ and $c_{2,i}\in[1,6/5)$ for all $i\in[k]$. We are ready to state the second main result.

\begin{theorem}\label{MPTC}
Assume $(X,\mathcal{B},\mu)$ is a $\sigma$-finite measure space and $\{S_i:\,i\in[k]\}$ is a family of invertible $\mu$-invariant commuting transformations. Then for any $p\in (1,\infty)$ and any $f\in L^p_{\mu}(X)$ we have that 
\[
\lim_{\min\{N_1,\dotsc,N_k\}\to\infty}\frac{1}{\prod_{i=1}^d |B_i\cap[1,N_i]|}\sum_{l \in \prod_{i=1}^d B_i\cap[1,N_i]}f \circ S^{(l)}(x)
\text{ exists for $\mu$-a.e. $x\in X$,} 
\]
where $S^{(l_1,\dotsc,l_k)}=S^{l_1}_1\circ\dots\circ S^{l_k}_k$.
\end{theorem}
We make some brief historical remarks. In 1951 Dunford \cite{DF} and Zygmund \cite{ZG}  independently showed that given a $\sigma$-finite measure space $(X,\mathcal{B},\mu)$ and a family of $\mu$-invariant transformations $\{T_i:\,i\in[k]\}$, for any $p\in (1,\infty)$ and any $f\in L^p_{\mu}(X)$, we have
\[
\frac{1}{N_1\cdots N_k}\sum_{l \in \prod_{i=1}^d [1,N_i]}f(T_1^{l_1}\cdots T_k^{l_k} x)
\quad\text{converges $\mu$-a.e. on $X$ and in $L^p$ norm as $\min\{N_1,\dotsc,N_k\}\to\infty$.}
\]
For $k\ge 2$, pointwise convergence fails on $L^1$. Motivated by that observation and after his seminal work on pointwise ergodic theory \cite{bg2,bg3,BET}, Bourgain showed that for any $p\in (1,\infty)$ and any $f\in L^p_{\mu}(X)$, we have
\[
\frac{1}{N_1\cdots N_k}\sum_{l \in \prod_{i=1}^d [1,N_i]}f\Big(T_1^{P_1(l_1)}\cdots T_k^{P_k(l_k)} x\Big)\quad\text{
converges $\mu$-a.e. on $X$ as $\min\{N_1,\dotsc,N_k\}\to\infty$,}
\]
where $P_1,\dotsc, P_k\in \mathbb{Z}[x]$, $P_1(0)=\dots=P_k(0)=0$ and $\{T_i:\,i\in[k]\}$ is a family of commuting and invertible $\mu$-invariant transformations. In contrast to Dunford and Zygmund's result, the commutativity assumption turns out to be indispensable for the polynomial case. For a more thorough exposition on the matter we refer the reader to Section~1.2 in \cite{MOE} as well as the introduction from \cite{MPVBF}, see page 3. In the spirit of the above, Theorem~$\ref{MPTC}$ establishes the multi-parameter result for orbits along sets of the form $B$. For example, for appropriate choices of parameters, Theorem~$\ref{MPTC}$ implies that for any $p\in (1,\infty)$ and any $f\in L^p_{\mu}(X)$, we have
\[
\frac{1}{N_1\cdots N_k}\sum_{l \in \prod_{i=1}^d [1,N_i]}f\Big(T_1^{\lfloor l_1^{c_1}\rfloor}\cdots T_k^{\lfloor l_k^{c_k}\rfloor} x\Big)
\quad\text{
converges $\mu$-a.e. on $X$ as $\min\{N_1,\dotsc,N_k\}\to\infty$,}
\]
where $c_1,\dotsc,c_k\in(1,6/5)$ and $\{T_i:\,i\in[k]\}$ is a family of commuting and invertible $\mu$-invariant transformations.

Using an abstract multi-parameter oscillation result from \cite{MOE}, we reduce the task of proving the above theorem to showing the following useful quantitative uniform estimates, which may  be of independent interest.

\begin{theorem}[Uniform 2-oscillation and vector-valued maximal estimates]\label{OM}Assume $c_1\in[1,2)$, $c_2\in[1,6/5)$ and $B$ as in the introduction. Assume $(X,\mathcal{B},\mu)$ is a $\sigma$-finite measure space and $T$ is an invertible $\mu$-invariant transformation. Let
\[
M_{B,t}f(x)=\frac{1}{|B\cap [1,t]|}\sum_{n\in B\cap [1,t]}f(T^nx)
\]
Then for any $p\in (1,\infty)$, there exists a constant $C_p$ such that 
\begin{equation}\label{U2OE}
\sup_{J\in\mathbb{N}}\sup_{I\in\mathfrak{S}_J(\mathbb{N})}\|O^2_{I,J}(M_{B,t}f:t\in\mathbb{N})\|_{L_{\mu}^p(X)}\le C_p\|f\|_{L_{\mu}^p(X)}\text{  for any $f\in L_{\mu}^p(X)$}
\end{equation}
and such that  
\begin{equation}\label{VVMFE}
\Big\|\Big(\sum_{j\in\mathbb{Z}}\big(\sup_{t\in\mathbb{N}}M_{B,t}|f_j|\big)^2\Big)^{1/2}\Big\|_{L_{\mu}^p(X)}\le C_p\Big\|\big(\sum_{j\in\mathbb{Z}}|f_j|^2\big)^{1/2}\Big\|_{L_{\mu}^p(X)}\text{ for any $(f_j)_{j\in\mathbb{Z}}\in L_{\mu}^p\big(\mathbb{Z};\ell^2(\mathbb{Z})\big)$}\text{.}
\end{equation}
\end{theorem}

We briefly note that for any $Y\subseteq X\subseteq \mathbb{R}$, with $|Y|>2$, we have that 
\[
\mathfrak{S}_J(X)=\{\{I_0<\dots<I_J\}\subseteq X\}\text{ and }O^2_{I,J}(a_t(x):\,t\in Y)=\Big(\sum_{j=0}^{J-1}\sup_{t\in [I_j,I_{j+1})\cap Y}|a_t(x)-a_{I_j}(x)|^2 \Big)^{1/2}\text{.}
\]
For the definition of multi-parameter oscillations as well as the basic properties of oscillations we refer the reader to section~2 from \cite{MOE}.

We now comment on the proof of Theorem~$\ref{OM}$. Again, Calder\'on's transference principle suggests that it suffices to establish these estimates for the integer shift system. Ultimately, those estimates are derived from the analogous ones for the standard discrete Hardy--Littlewood averaging operator. For the vector-valued maximal inequality we use the exponential sum estimates of Lemma~$\ref{TrEst}$ together with the fact that $\psi$ behaves ``like a constant'' in dyadic blocks in order to eventually be able to use the corresponding estimates for the Hardy--Littlewood averaging operator (for example see Theorem~1 in \cite{SF} or Theorem~C in \cite{MVV}). The situation is much more complicated for the oscillations. We follow the strategy from \cite{JI} and \cite{PZ}, and we break our analysis into short and long oscillations, and instead of opting to handle as our ``long oscillations'' the rather natural choice $\{2^n:\,n\in\mathbb{N}_0\}$, we choose a much denser set, namely on $\{\lfloor 2^{n^{\tau}}\rfloor:\,n\in\mathbb{N}_0\}$, for $\tau$ small. This affords us to bound the short oscillations straightforwardly. Loosely speaking, the long oscillations are treated in a similar manner to the vector-valued maximal inequality, but the fact that the 2-oscillations are not a positive operator makes the use of the fact that $\psi$ behaves nicely in dyadic blocks difficult. Here, we adapt the argument from section~5 in \cite{UA} to our oscillation setting in order to compare averages with different weights. Again, we use the uniform oscillation estimates for the discrete Hardy--Littlewood averaging operator to conclude (which one may find for example in \cite{UOI} or \cite{UOIAO}). Finally, we mention that the exponential sum estimates help us understand some error terms on $\ell^2$, and Riesz--Thorin interpolation together with trivial bounds coming from the fact that we deal with averaging operators help us establish the corresponding $\ell^p$ bounds.
  
\subsection{Notation}
We denote by $C$ a positive constant that may change from occurrence to occurrence. If $A,B$ are two non-negative quantities, we write $A\lesssim B$ or $B \gtrsim A$  to denote that there exists a positive constant $C$ such that $A\le C B$. Whenever $A\lesssim B$ and $A\gtrsim B$ we write $A\simeq B$. For two complex-valued functions $f,g$ we write $f\sim g$ to denote that $\lim_{x\to\infty}\frac{f(x)}{g(x)}=1$. We denote the average value of a function  $f\colon\mathbb{Z}\to\mathbb{C}$ over a finite set $Q\subseteq \mathbb{Z}$ by $[f]_Q=\frac{1}{|Q|}\sum_{x\in Q}f(x)$. For any natural number $N$, we let $[N]=\{1,2,\dotsc,N\}$.

\section{Basic Properties of the sets $B$} 
In this section we collect some useful properties of the sets $B$. We begin by stating an exponential sum estimate proven in \cite{HLMP}. Here we fix two constants $c_1,c_2$ such that $c_1\in[1,2)$ and $c_2\in[1,6/5)$, as well as $h_1$, $h_2$ in $\mathcal{R}_{c_1}$ and $\mathcal{R}_{c_2}$ respectively and $\psi$, $\gamma_1=1/c_1$, $\gamma_2=1/c_2$ as in the introduction and all the implied constants may depend on them. We note that we use the basic properties of those functions as described in Lemma~2.6  and Lemma~2.14 from \cite{W11} without further mention.

\begin{lemma}\label{TrEst}
Assume $\chi>0$ is such that $(1-\gamma_1)+3(1-\gamma_2)+6\chi<1$. Then there exists a real number $\chi'>0$ such that
\begin{equation}\label{EXT}
\sum_{n\in B\cap[N]} \psi(n)^{-1} e^{2\pi i n\xi}= \sum_{n\in [N]} e^{2\pi i n\xi} +O(N^{1-\chi-\chi'})
\end{equation}
as well as
\begin{equation}\label{EXT2}
\sum_{n\in B\cap [N]}e^{2\pi i n\xi}= \sum_{n\in [N]}\psi(n) e^{2\pi i n\xi} +O(\varphi_2(N)N^{-\chi-\chi'})
\end{equation}
where the implied constant does not depend on $\xi$ or $N$.
\end{lemma}

\begin{proof} We note that one is derived from the other using summation by parts. The proof can be found in page 6 as well as in Lemma~3.2 in \cite{HLMP}. 
\end{proof}

\begin{lemma}\label{BVBI}
If $\varphi_2\lesssim \varphi_1$, then $B$ does not contain arbitrarily long intervals in $\mathbb{Z}$.
\end{lemma}
\begin{proof}
Let us assume that $\varphi_2\lesssim \varphi_1$ and $\{n,n+1,\dotsc,n+l-1\}\subseteq B$. We wish to bound $l$. Let us notice that $B$ may be partitioned as follows
\[
B=\bigcup_{m\in\mathbb{N}}B_m
\]where $B_m=\{\,n\in\mathbb{N}\,:\,0\le \varphi_1(n)-m<\psi(n)\,\}$, note that $B_m\cap B_k=\emptyset$ for $m\neq k$. For sufficiently large $m,k$ with $m<k$, we have that $dist(B_m,B_k)\ge 2$ since if we assume for the sake of a contradiction that $dist(B_m,B_k)=1$ then there exists $n\in B_m$ such that and $n+1\in B_k$. But then
\[
0\le \varphi_1(n)-m<\psi(n)\text{ and }0\le \varphi_1(n+1)-k<\psi(n+1)
\]
and thus
\[
(k-m)-\psi(n)<\varphi_1(n+1)-\varphi_1(n)<(k-m)+\psi(n+1)
\]
thus 
\[
\varphi_1(n+1)-\varphi_1(n)>1/2
\]
and by the Mean Value Theorem there exists $\xi_n\in (n,n+1)$ such that $\varphi_1'(\xi_n)>1/2$, which will be a contradiction for sufficiently large $n$, since $\varphi_1'(x)\to 0$, as $x\to\infty$. Thus, if we ignore some first few terms of the set $B$, then the fact that $\{n,n+1,\dotsc,n+l-1\}\subseteq B$, together with our previous observation, imply that there exists an $m\in\mathbb{N}$ such that $\{n,n+1,\dotsc,n+l-1\}\subseteq B_m$. Combining $n,n+l-1\in B_m$, we obtain
\[
\varphi_1(n+l-1)-\varphi_1(n)<\psi(n+l)
\]
which by the Mean Value Theorem becomes
\[
(l-1)\varphi_1'(n+l)\le (l-1)\varphi_1'(\xi_{n,l})<\psi(n+l)\lesssim \varphi_2'(n+l) 
\]
and thus $(l-1)\lesssim \frac{\varphi_2'(n+l)}{\varphi_1'(n+l)}\lesssim \frac{\varphi_2(n+l)}{\varphi_1(n+l)}\lesssim 1$, since $\varphi_2\lesssim \varphi_1$. Thus the set $B$ does not contain arbitrarily long intervals. 
\end{proof}
\section{Uniform 2-oscillation Estimates}Here we wish to prove the first half of Proposition~$\ref{OM}$, namely to establish the estimate $(\ref{U2OE})$. Similarly to the previous section, we fix $B$ as in the introduction with $(c_1,c_2)\in[1,2)\times[1,6/5)$ and all the implied constants may depend on them. By the Calder\'on Transference Principle, in order  to establish that estimate for any $\sigma$-finite measure preserving system, it suffices to establish it for the integer shift system, namely for $(\mathbb{Z},\mathcal{P}(\mathbb{Z}),|\cdot|,S)$ where $S$ is the shift map $S(x)=x-1$ and the $|\cdot|$ is the counting measure. To simplify the notation, we let $B_t=B\cap [1,t]$ and write
\[
M_{t}f(x)=\frac{1}{|B_t|}\sum_{n\in B_t}f(x-n)
\]
 We therefore wish to show that for any $p\in(1,\infty)$, there exists a constant $C_p$ such that
 
\begin{equation}\label{U2OEI}
\sup_{J\in\mathbb{N}_0}\sup_{I\in\mathfrak{S}_J(\mathbb{N})}\|O^2_{I,J}(M_tf:t\in\mathbb{N})\|_{\ell^p(\mathbb{Z})}\le C_p\|f\|_{\ell^p(\mathbb{Z})}\text{  for any $f\in \ell^p(\mathbb{Z})$}
\end{equation}
To establish this result, we break the 2-oscillations into short and long ones. To do that, we need to carefully choose some parameters first. Let $p_0\in(1,\infty)$ be such that $p\in(p_0,p_0')$, let $\tau\in(0,\min\big\{\frac{p_0-1}{2},\frac{1}{2}\big\})$ and let $\mathbb{D}_{\tau}=\{2^{n^{\tau}}:\,n\in\mathbb{N}_0\}$. It is not difficult to see that
 \begin{multline}\label{SL}
 \sup_{J\in\mathbb{N}}\sup_{I\in\mathfrak{S}_J(\mathbb{N})}\|O^2_{I,J}(M_{t}f:t\in\mathbb{N})\|_{\ell^p(\mathbb{Z})}\lesssim
 \sup_{J\in\mathbb{N}}\sup_{I\in\mathfrak{S}_J(\mathbb{D}_{\tau})}\|O^2_{I,J}(M_{t}f:t\in\mathbb{D}_{\tau})\|_{\ell^p(\mathbb{Z})}+
 \\
 +\bigg\|\bigg(\sum_{n=0}^{\infty}V^2\Big(M_tf:\,t\in\big[2^{n^{\tau}},2^{(n+1)^{\tau}}\big)\Big)^2\bigg)^{1/2}\bigg\|_{\ell^p(\mathbb{Z})}
 \end{multline}
 where 
 \[
V^2\Big(M_tf(x):\,t\in\big[2^{n^{\tau}},2^{(n+1)^{\tau}}\big)\Big)=\sup_{J\in\mathbb{N}}\sup_{\substack{t_0<\dots<t_J\\t_j\in\big[2^{n^{\tau}},2^{(n+1)^{\tau}}\big)}}\Big(\sum_{j=0}^{J-1}|M_{t_{j+1}}f(x)-M_{t_j}f(x)|^2 \Big)^{1/2}\text{,}
\]
see \cite{MOE} page 17. One may adapt the argument appearing in Lemma~1.3 in \cite{JEFM} to establish this.

We deal with the second term of $(\ref{SL})$. Note that
\[
\bigg\|\bigg(\sum_{n=0}^{\infty}V^2\Big(M_tf:\,t\in\big[2^{n^{\tau}},2^{(n+1)^{\tau}}\big)\Big)^2\bigg)^{1/2}\bigg\|_{\ell^p(\mathbb{Z})}=
\]
\[
\bigg\|\bigg(\sum_{n=0}^{\infty}\bigg(\sup_{J\in\mathbb{N}}\sup_{\substack{t_0<\dots t_J\\t_j\in\big[2^{n^{\tau}},2^{(n+1)^{\tau}}\big)}}\Big(\sum_{j=0}^{J-1}|M_{t_{j+1}}f-M_{t_{j}}f|^2\Big)^{1/2}\bigg)^{2}\bigg)^{1/2}\bigg\|_{\ell^p(\mathbb{Z})}\lesssim
\] 
\[
\bigg\|\bigg(\sum_{n=1}^{\infty}\bigg(\sup_{J\in\mathbb{N}}\sup_{\substack{t_0<\dots t_J\\t_j\in\big[2^{(n-1)^{\tau}},2^{(n+2)^{\tau}}\big)\cap B}}\Big(\sum_{j=0}^{J-1}|M_{t_{j+1}}f-M_{t_{j}}f|\Big)\bigg)^{2}\bigg)^{1/2}\bigg\|_{\ell^p(\mathbb{Z})}
\] 
where we have used the fact that $\|\cdot\|_{\ell^2}\le\|\cdot\|_{\ell^1}$, together with the fact that $M_tf(x)=M_sf(x)$ whenever $B_t=B_s$. For any $n
\in\mathbb{N}$, let $\big\{\beta_{m}^{(n)}:\,m\in\{0,\dotsc,l_n\}\big\}$ be an increasing enumeration of $\big[2^{(n-1)^{\tau}},2^{(n+2)^{\tau}}\big)\cap B$. Then we use the triangular inequality to bound the last expression by 
\[
\bigg\|\bigg(\sum_{n=1}^{\infty}\bigg(\Big(\sum_{m=1}^{l_n}|M_{\beta_{m}^{(n)}}f-M_{\beta_{m-1}^{(n)}}f|\Big)\bigg)^{2}\bigg)^{1/2}\bigg\|_{\ell^p(\mathbb{Z})}
\]
Let $K_t(x)=\frac{1}{|B_t|}\sum_{n\in B_t}\delta_n(x)$, where $\delta_n(x)=1_{\{n\}}(x)$, and note that $M_tf(x)=K_t*f(x)$. Thus we rewrite the expression above as
\[
\bigg\|\bigg(\sum_{n=1}^{\infty}\bigg(\Big(\sum_{m=1}^{l_n}|(K_{\beta_{m}^{(n)}}-K_{\beta_{m-1}^{(n)}})*f|\Big)\bigg)^{2}\bigg)^{1/2}\bigg\|_{\ell^p(\mathbb{Z})}
\] 
We firstly consider the case $p\in(2,\infty)$, we get
\[
\bigg\|\bigg(\sum_{n=0}^{\infty}V^2\Big(M_tf:\,t\in\big[2^{n^{\tau}},2^{(n+1)^{\tau}}\big)\Big)^2\bigg)^{1/2}\bigg\|_{\ell^p(\mathbb{Z})}\lesssim\bigg\|\bigg(\sum_{n=1}^{\infty}\bigg(\Big(\sum_{m=1}^{l_n}|K_{\beta_{m}^{(n)}}-K_{\beta_{m-1}^{(n)}}|*|f|\Big)\bigg)^{2}\bigg)^{1/2}\bigg\|_{\ell^p(\mathbb{Z})}=
\]
\[
\bigg(\sum_{x\in\mathbb{Z}}\bigg(\sum_{n=1}^{\infty}\bigg(\Big(\sum_{m=1}^{l_n}|K_{\beta_{m}^{(n)}}-K_{\beta_{m-1}^{(n)}}|*|f|(x)\Big)\bigg)^{2}\bigg)^{p/2}\bigg)^{1/p}\le
\]
\[
\bigg(\sum_{n=1}^{\infty}\bigg(\sum_{m=1}^{l_n}\Big(\sum_{x\in\mathbb{Z}}\Big(|K_{\beta_{m}^{(n)}}-K_{\beta_{m-1}^{(n)}}|*|f|(x)\Big)^p\Big)^{1/p}\bigg)^{2}\bigg)^{1/2}=
\]
\[
\bigg(\sum_{n=1}^{\infty}\bigg(\sum_{m=1}^{l_n}\big\||K_{\beta_{m}^{(n)}}-K_{\beta_{m-1}^{(n)}}|*|f|\big\|_{\ell^p(\mathbb{Z})}\bigg)^{2}\bigg)^{1/2}\le
\]
\[
\bigg(\sum_{n=1}^{\infty}\bigg(\sum_{m=1}^{l_n}\big\|K_{\beta_{m}^{(n)}}-K_{\beta_{m-1}^{(n)}}\big\|_{\ell^1(\mathbb{Z})}\|f\|_{\ell^p(\mathbb{Z})}\bigg)^{2}\bigg)^{1/2}=\bigg(\sum_{n=1}^{\infty}\bigg(\sum_{m=1}^{l_n}\big\|K_{\beta_{m}^{(n)}}-K_{\beta_{m-1}^{(n)}}\big\|_{\ell^1(\mathbb{Z})}\bigg)^{2}\bigg)^{1/2}\|f\|_{\ell^p(\mathbb{Z})}
\]
where we have used Minkowski's inequality for $p/2>1$ and $p>1$, and then Young's convolution inequality.
In the case where $p\in(1,2]$ we note 
\[
\bigg\|\bigg(\sum_{n=0}^{\infty}V^2\Big(M_tf:\,t\in\big[2^{n^{\tau}},2^{(n+1)^{\tau}}\big)\Big)^2\bigg)^{1/2}\bigg\|_{\ell^p(\mathbb{Z})}\lesssim\bigg\|\bigg(\sum_{n=1}^{\infty}\bigg(\Big(\sum_{m=1}^{l_n}|K_{\beta_{m}^{(n)}}-K_{\beta_{m-1}^{(n)}}|*|f|\Big)\bigg)^{2}\bigg)^{1/2}\bigg\|_{\ell^p(\mathbb{Z})}\le
\]
\[
\bigg\|\bigg(\sum_{n=1}^{\infty}\bigg(\Big(\sum_{m=1}^{l_n}|K_{\beta_{m}^{(n)}}-K_{\beta_{m-1}^{(n)}}|*|f|\Big)\bigg)^{p}\bigg)^{1/p}\bigg\|_{\ell^p(\mathbb{Z})}\le
\]
\[
\bigg(\sum_{x\in\mathbb{Z}}\sum_{n=1}^{\infty}\bigg(\Big(\sum_{m=1}^{l_n}|K_{\beta_{m}^{(n)}}-K_{\beta_{m-1}^{(n)}}|*|f|(x)\Big)\bigg)^{p}\bigg)^{1/p}\le
\]
\[
\bigg(\sum_{n=1}^{\infty}\bigg(\sum_{m=1}^{l_n}\Big(\sum_{x\in\mathbb{Z}}\Big(|K_{\beta_{m}^{(n)}}-K_{\beta_{m-1}^{(n)}}|*|f|(x)\Big)^p\Big)^{1/p}\bigg)^{p}\bigg)^{1/p}=
\]
\[
\bigg(\sum_{n=1}^{\infty}\bigg(\sum_{m=1}^{l_n}\big\||K_{\beta_{m}^{(n)}}-K_{\beta_{m-1}^{(n)}}|*|f|\big\|_{\ell^p(\mathbb{Z})}\bigg)^{p}\bigg)^{1/p}\le
\]
\[
\bigg(\sum_{n=1}^{\infty}\bigg(\sum_{m=1}^{l_n}\big\|K_{\beta_{m}^{(n)}}-K_{\beta_{m-1}^{(n)}}\big\|_{\ell^1(\mathbb{Z})}\|f\|_{\ell^p(\mathbb{Z})}\bigg)^{p}\bigg)^{1/p}=\bigg(\sum_{n=1}^{\infty}\bigg(\sum_{m=1}^{l_n}\big\|K_{\beta_{m}^{(n)}}-K_{\beta_{m-1}^{(n)}}\big\|_{\ell^1(\mathbb{Z})}\bigg)^{p}\bigg)^{1/p}\|f\|_{\ell^p(\mathbb{Z})}
\]
where we have used Minkowski's inequality and Young's convolution inequality.
Combining the two cases gives
\begin{equation}\label{SFEst}
\bigg\|\bigg(\sum_{n=0}^{\infty}V^2\Big(M_tf:\,t\in\big[2^{n^{\tau}},2^{(n+1)^{\tau}}\big)\Big)^2\bigg)^{1/2}\bigg\|_{\ell^p(\mathbb{Z})}\lesssim \bigg(\sum_{n=1}^{\infty}\bigg(\sum_{m=1}^{l_n}\big\|K_{\beta_{m}^{(n)}}-K_{\beta_{m-1}^{(n)}}\big\|_{\ell^1(\mathbb{Z})}\bigg)^{q}\bigg)^{1/q}\|f\|_{\ell^p(\mathbb{Z})}
\end{equation}
where $q=\min\{2,p\}$.
We focus on $\big\|K_{\beta_{m}^{(n)}}-K_{\beta_{m-1}^{(n)}}\big\|_{\ell^1(\mathbb{Z})}$, note that
\[
\big\|K_{\beta_{m}^{(n)}}-K_{\beta_{m-1}^{(n)}}\big\|_{\ell^1(\mathbb{Z})}=\sum_{x\in B_{\beta_{m-1}^{(n)}}}\frac{1}{|B_{\beta_{m-1}^{(n)}}|}-\frac{1}{|B_{\beta_{m}^{(n)}}|}+\sum_{x\in B_{\beta_{m}^{(n)}}\setminus B_{\beta_{m-1}^{(n)}}}\frac{1}{|B_{\beta_{m}^{(n)}}|}=
\]
\[
\frac{|B_{\beta_{m}^{(n)}}|-|B_{\beta_{m-1}^{(n)}}|}{|B_{\beta_{m}^{(n)}}|}+\frac{|B_{\beta_{m}^{(n)}}\setminus B_{\beta_{m-1}^{(n)}}|}{|B_{\beta_{m}^{(n)}}|}=2\frac{|B_{\beta_{m}^{(n)}}|-|B_{\beta_{m-1}^{(n)}}|}{|B_{\beta_{m}^{(n)}}|}
\]
thus
\[
\sum_{m=1}^{l_n}\big\|K_{\beta_{m}^{(n)}}-K_{\beta_{m-1}^{(n)}}\big\|_{\ell^1(\mathbb{Z})}\lesssim \sum_{m=1}^{l_n}\frac{|B_{\beta_{m}^{(n)}}|-|B_{\beta_{m-1}^{(n)}}|}{|B_{\beta_{m}^{(n)}}|}\lesssim\frac{|B_{2^{(n+2)^{\tau}}}|-|B_{2^{(n-1)^{\tau}}}|}{|B_{2^{(n-1)^{\tau}}}|}
\]
We know that there exists $\varepsilon>0$ such that $|B_t|=\varphi_2(t)(1+O(t^{-\varepsilon}))$ (see \cite{HLMP}, page  5), thus we get
\[
\sum_{m=1}^{l_n}\big\|K_{\beta_{m}^{(n)}}-K_{\beta_{m-1}^{(n)}}\big\|_{\ell^1(\mathbb{Z})}\lesssim \frac{\varphi_2(2^{(n+2)^{\tau}})-\varphi_2(2^{(n-1)^{\tau}})+\varphi_2(2^{(n+2)^{\tau}})O(2^{-\varepsilon (n-1)^{\tau}})}{\varphi_2(2^{(n-1)^{\tau}})}
\]
Note that $\frac{\varphi_2(2^{(n+2)^{\tau}})2^{-\varepsilon (n-1)^{\tau}}}{\varphi_2(2^{(n-1)^{\tau}})}\le \frac{\varphi_2(2^{(n-1)^{\tau}+3^{\tau}})2^{-\varepsilon (n-1)^{\tau}}}{\varphi_2(2^{(n-1)^{\tau}})}\lesssim2^{-\varepsilon (n-1)^{\tau}}$, since $0<\tau \le 1/2$, $\varphi_2$ is increasing and $\varphi_2(4x)\lesssim \varphi_2(x)$. We also note that by the Mean Value Theorem for $f(x)=\varphi_2(2^{x^{\tau}})$ on the interval $[n-1,n+2]$ we get that there exists $x_n\in(n-1,n+2)$ such that
\[
\frac{\varphi_2(2^{(n+2)^{\tau}})-\varphi_2(2^{(n-1)^{\tau}})}{3}=\frac{f(n+2)-f(n-1)}{3}=f'(x_n)=\varphi_2'(2^{x_n^{\tau}})2^{x_n^{\tau}}\log(2)\tau x_n^{\tau-1}\lesssim
\]
\[
\varphi_2(2^{x_n^{\tau}}) (n-1)^{\tau-1}\le \varphi_2(2^{(n+2)^{\tau}}) (n-1)^{\tau-1} \le \varphi_2(2^{(n-1)^{\tau}+3^{\tau}})(n-1)^{\tau-1} \lesssim \varphi_2(2^{(n-1)^{\tau}})(n-1)^{\tau-1}
\]
Thus 
\[
\sum_{m=1}^{l_n}\big\|K_{\beta_{m}^{(n)}}-K_{\beta_{m-1}^{(n)}}\big\|_{\ell^1(\mathbb{Z})}\lesssim (n-1)^{\tau-1}+2^{-\varepsilon (n-1)^{\tau}}
\]
which implies
\[
\bigg(\sum_{n=1}^{\infty}\bigg(\sum_{m=1}^{l_n}\big\|K_{\beta_{m}^{(n)}}-K_{\beta_{m-1}^{(n)}}\big\|_{\ell^1(\mathbb{Z})}\bigg)^{q}\bigg)^{1/q}\lesssim 
\]
\[
1+\bigg(\sum_{n=2}^{\infty}(n-1)^{q(\tau-1)}\bigg)^{1/q}+\bigg(\sum_{n=1}^{\infty}2^{-\varepsilon q(n-1)^{\tau}}\bigg)^{1/q}
\]
Note that if $p>2$ then $q=2$ and then  $q(1-\tau)>2(1-1/2)=1$ thus the first sum converges. Similarly, if $p\le 2$, then $q=p$. Note that $\tau<(p_0-1)/2<(p_0-1)/p_0=1/p_0'$, but then $q(1-\tau)>p(1-1/p_0')=p/p_0>1$ as desired. In either case, the first series is summable. The second series is also summable, since for example $2^{-q\varepsilon n^{\tau}}\lesssim_{p,\tau} n^{-2}$, and thus ($\ref{SFEst}$) becomes
\[
\bigg\|\bigg(\sum_{n=0}^{\infty}V^2\Big(M_tf:\,t\in\big[2^{n^{\tau}},2^{(n+1)^{\tau}}\big)\Big)^2\bigg)^{1/2}\bigg\|_{\ell^p(\mathbb{Z})}\lesssim_{p,\tau}\|f\|_{\ell^p(\mathbb{Z})}
\]

To establish the desired estimate, it remains to bound the first term of ($\ref{SL}$). Note that
\[ \sup_{J\in\mathbb{N}}\sup_{I\in\mathfrak{S}_J(\mathbb{D}_{\tau})}\|O^2_{I,J}(M_{t}f:t\in\mathbb{D}_{\tau})\|_{\ell^p(\mathbb{Z})}=
\sup_{J\in\mathbb{N}_0}\sup_{I\in \mathfrak{S}_J(\mathbb{N}_0)}\|O^2_{I,J}\big(M_{2^{n^{\tau}}}f:\,n\in\mathbb{N}_0\big)\|_{\ell^p(\mathbb{Z})}\text{.}
\]
We wish to show
\begin{equation}\label{LOBLP}
\sup_{J\in\mathbb{N}}\sup_{I\in \mathfrak{S}_J(\mathbb{N}_0)}\|O^2_{I,J}\big(M_{2^{n^{\tau}}}f:\,n\in\mathbb{N}_0\big)\|_{\ell^p(\mathbb{Z})}
\lesssim_{p,\tau} \|f\|_{\ell^p(\mathbb{Z})}
\end{equation}

\begin{proof}[Proof of the estimate $(\ref{LOBLP})$]We introduce some auxiliary averaging operators. Let
\[
H_{t}f(x)=\frac{1}{t}\sum_{1\le s\le t}f(x-s)\text{ and }A_{t}f(x)=\frac{1}{|B_t|}\sum_{1\le s\le t}\psi(s)f(x-s)=(L_t*f)(x)
\]
where $L_t(x)=\frac{1}{|B_t|}\sum_{1\le s\le t}\psi(s)\delta_{s}(x)$.
We may compare $M_t$ with $A_t$ as follows
\[
\sup_{J\in\mathbb{N}_0}\sup_{I\in \mathfrak{S}_J(\mathbb{N}_0)}\|O^2_{I,J}\big(M_{2^{n^{\tau}}}f:\,n\in\mathbb{N}_0\big)\|_{\ell^p(\mathbb{Z})}\le
\]
\[
\sup_{J\in\mathbb{N}_0}\sup_{I\in \mathfrak{S}_J(\mathbb{N}_0)}\|O^2_{I,J}\big(M_{2^{n^{\tau}}}f-A_{2^{n^{\tau}}}f:\,n\in\mathbb{N}_0\big)\|_{\ell^p(\mathbb{Z})}+\sup_{J\in\mathbb{N}_0}\sup_{I\in \mathfrak{S}_J(\mathbb{N}_0)}\|O^2_{I,J}\big(A_{2^{n^{\tau}}}f:\,n\in\mathbb{N}_0\big)\|_{\ell^p(\mathbb{Z})}\lesssim
\]
\begin{equation}\label{2terms}
\Big\|\Big(\sum_{n\in\mathbb{N}_0}|M_{2^{n^{\tau}}}f-A_{2^{n^{\tau}}}f|^2\Big)^{1/2} \Big\|_{\ell^p(\mathbb{Z})}+\sup_{J\in\mathbb{N}_0}\sup_{I\in \mathfrak{S}_J(\mathbb{N}_0)}\|O^2_{I,J}\big(A_{2^{n^{\tau}}}f:\,n\in\mathbb{N}_0\big)\|_{\ell^p(\mathbb{Z})}
\end{equation}
The first term of the expression ($\ref{2terms}$) will be bounded using the Lemma~$\ref{TrEst}$ and interpolation. More specifically, we start with $p=2$, and we note
\[
\Big\|\Big(\sum_{n\in\mathbb{N}_0}|M_{2^{n^{\tau}}}f-A_{2^{n^{\tau}}}f|^2\Big)^{1/2} \Big\|_{\ell^2(\mathbb{Z})}=\Big(\sum_{n\in\mathbb{N}_0}\|M_{2^{n^{\tau}}}f-A_{2^{n^{\tau}}}f\|^2_{\ell^2(\mathbb{Z})} \Big)^{1/2}
\] 
and for each $n\in\mathbb{N}_0$, we have
\[
\|M_{2^{n^{\tau}}}f-A_{2^{n^{\tau}}}f\|_{\ell^2(\mathbb{Z})}=\|(K_{2^{n^{\tau}}}-L_{2^{n^{\tau}}})*f\|_{\ell^2(\mathbb{Z})}=\|(\widehat{K}_{2^{n^{\tau}}}-\widehat{L}_{2^{n^{\tau}}})\widehat{f}\|_{L^2(\mathbb{T})}\le \|\widehat{K}_{2^{n^{\tau}}}-\widehat{L}_{2^{n^{\tau}}}\|_{L^{\infty}(\mathbb{T})}\|f\|_{\ell^2(\mathbb{Z})}
\]
 and note that there exists $\chi>0$ such that for any $\xi\in\mathbb{T}$ we get
 \[
 |\widehat{K}_{2^{n^{\tau}}}(\xi)-\widehat{L}_{2^{n^{\tau}}}(\xi)|=\bigg| \frac{1}{|B_{2^{n^{\tau}}}|}\sum_{s\in B_{2^{n^{\tau}}}}e^{2\pi i s\xi}-\frac{1}{|B_{2^{n^{\tau}}}|}\sum_{1\le s\le 2^{n^{\tau}}}\psi(s) e^{2\pi i s\xi}\bigg|\lesssim\frac{\varphi_2(2^{n^{\tau}})2^{-\chi n^{\tau}}}{|B_{2^{n^{\tau}}}|}\lesssim 2^{-\chi n^{\tau}}
\]
and thus
\[
\|M_{2^{n^{\tau}}}f-A_{2^{n^{\tau}}}f\|_{\ell^2(\mathbb{Z})}\lesssim 2^{-\chi n^\tau}\|f\|_{\ell^2(\mathbb{Z})}
\]
which gives
\[
\Big\|\Big(\sum_{n\in\mathbb{N}_0}|M_{2^{n^{\tau}}}f-A_{2^{n^{\tau}}}f|^2\Big)^{1/2} \Big\|_{\ell^2(\mathbb{Z})}\lesssim \Big(\sum_{n\in\mathbb{N}_0}2^{-2\chi n^{\tau}}\Big)^{1/2}\|f\|_{\ell^2(\mathbb{Z})}\lesssim_{\tau}\|f\|_{\ell^2(\mathbb{Z})}
\]
For the case of $p\neq 2$, firstly, let us assume that $p\in(2,\infty)$. Note that there exists a positive constant $C$ such that
\[
\|M_{2^{n^{\tau}}}f-A_{2^{n^{\tau}}}f\|_{\ell^{p_0'}(\mathbb{Z})}\le 
\|M_{2^{n^{\tau}}}f\|_{\ell^{p_0'}(\mathbb{Z})}+\|A_{2^{n^{\tau}}}f\|_{\ell^{p_0'}(\mathbb{Z})}\le C\|f\|_{\ell^{p_0'}(\mathbb{Z})}
\] since
\[
\bigg\|\frac{1}{|B_t|}\sum_{s\in B_t}f(\cdot-s)\bigg\|_{\ell^{p_0'}(\mathbb{Z})}\le \|f\|_{\ell^{p_0'}(\mathbb{Z})}\text{ and}
\]
\[
\bigg\|\frac{1}{|B_t|}\sum_{1\le s\le t}\psi(s)f(\cdot-s)\bigg\|_{\ell^{p_0'}(\mathbb{Z})}\le \frac{1}{|B_t|}\sum_{1\le s\le t}\psi(s)\|f\|_{\ell^{p_0'}(\mathbb{Z})}\lesssim \|f\|_{\ell^{p_0'}}
\]
where we have used Lemma~$\ref{TrEst}$ for $\xi=0$. We may choose $\theta\in (0,1)$ such that $\frac{1}{p}=\frac{\theta}{2}+\frac{1-\theta}{p_0'}$ and use Riesz--Thorin interpolation theorem. Since for any $n\in\mathbb{N}_0$ we have
\[
\|M_{2^{n^{\tau}}}f-A_{2^{n^{\tau}}}f\|_{\ell^{2}(\mathbb{Z})}\le C2^{-\chi n^{\tau}}\text{ and }\|M_{2^{n^{\tau}}}f-A_{2^{n^{\tau}}}f\|_{\ell^{p_0'}(\mathbb{Z})}\le C\|f\|_{\ell^{p_0'}(\mathbb{Z})}
\]
we interpolate to obtain
\[
\|M_{2^{n^{\tau}}}f-A_{2^{n^{\tau}}}f\|_{\ell^{p}(\mathbb{Z})}\le (C  2^{-\chi n^{\tau}})^{\theta}C^{1-\theta}\|f\|_{\ell^{p}(\mathbb{Z})}=C (2^{-\chi \theta})^{n^{\tau}} \|f\|_{\ell^{p}(\mathbb{Z})}
\]
We now note that since $p/2>1$ we have
\[
\Big\|\Big(\sum_{n\in\mathbb{N}_0}|M_{2^{n^{\tau}}}f-A_{2^{n^{\tau}}}f|^2\Big)^{1/2} \Big\|_{\ell^p(\mathbb{Z})}=\Big(\sum_{x\in\mathbb{Z}}\Big(\sum_{n\in\mathbb{N}_0}|M_{2^{n^{\tau}}}f(x)-A_{2^{n^{\tau}}}f(x)|^2\Big)^{p/2}\Big)^{1/p}\le
\]
\[
\Big(\sum_{n\in\mathbb{N}_0}\Big(\sum_{x\in\mathbb{Z}}|M_{2^{n^{\tau}}}f(x)-A_{2^{n^{\tau}}}f(x)|^{p}\Big)^{2/p}\Big)^{1/2}=\Big(\sum_{n\in\mathbb{N}_0}\|M_{2^{n^{\tau}}}f-A_{2^{n^{\tau}}}f\|_{\ell^{p}(\mathbb{Z})}^2\Big)^{1/2}\le
\]
\[C\Big(\sum_{n\in\mathbb{N}_0}(2^{-2\chi \theta})^{n^{\tau}}\Big)^{1/2}\|f\|_{\ell^{p}(\mathbb{Z})}\le C_{p,\tau}\|f\|_{\ell^{p}(\mathbb{Z})}
\]
The case of $p\in (1,2)$ is similar; we choose $\theta\in (0,1)$ such that $\frac{1}{p}=\frac{\theta}{2}+\frac{1-\theta}{p_0}$. Riesz–Thorin interpolation theorem yields the same estimate as before
\[
\|M_{2^{n^{\tau}}}f-A_{2^{n^{\tau}}}f\|_{\ell^{p}(\mathbb{Z})}\le C (2^{-\chi \theta})^{n^{\tau}} \|f\|_{\ell^{p}(\mathbb{Z})}
\]and we note that
\[
\Big\|\Big(\sum_{n\in\mathbb{N}_0}|M_{2^{n^{\tau}}}f-A_{2^{n^{\tau}}}f|^2\Big)^{1/2} \Big\|_{\ell^p(\mathbb{Z})}\le \Big\|\Big(\sum_{n\in\mathbb{N}_0}|M_{2^{n^{\tau}}}f-A_{2^{n^{\tau}}}f|^p\Big)^{1/p} \Big\|_{\ell^p(\mathbb{Z})}\le
\]
\[
\Big(\sum_{n\in\mathbb{N}_0}\|M_{2^{n^{\tau}}}f-A_{2^{n^{\tau}}}f\|_{\ell^p(\mathbb{Z})}^p\Big)^{1/p}\le C\Big(\sum_{n\in\mathbb{N}_0}(2^{-p\chi\theta })^{n^{\tau}}\Big)^{1/p}\|f\|_{\ell^p(\mathbb{Z})}\le C_{p,\tau}\|f\|_{\ell^p(\mathbb{Z})}
\]
We have appropriately bounded the first term of equation ($\ref{2terms}$). We now focus on the second term. We will reduce the 2-oscillation $\ell^p$ estimates for $A_{2^{n^{\tau}}}$ to the corresponding ones for the $C_{2^{n^{\tau}}}$. Firstly, the analysis of the 2-oscillations will be made easier if we adjust $A_t$ to the following very similar operator 
\[
D_t f(x)=\frac{1}{\sum_{1\le s\le t}\psi(s)}\sum_{1 \le s\le t}\psi(s)f(x-s)
\]
Note that
\[|A_{2^{n^{\tau}}}f(x)-D_{2^{n^{\tau}}}f(x)|=\bigg|\bigg(\frac{1}{|B_{2^{n^{\tau}}}|}-\frac{1}{\sum_{1\le s\le 2^{n^{\tau}}}\psi(s)}\bigg)\sum_{1\le s\le 2^{n^{\tau}}}\psi(s)f(x-s)\bigg|
\]
Thus
\[
\|A_{2^{n^{\tau}}}f-D_{2^{n^{\tau}}}f\|_{\ell^p(\mathbb{Z})}\le \frac{\big|\sum_{1\le s\le 2^{n^{\tau}}}\psi(s)-|B_{2^{n^{\tau}}}|\big|}{|B_{2^{n^{\tau}}}|\sum_{1\le s\le 2^{n^{\tau}}}\psi(s)}\sum_{1\le s\le 2^{n^{\tau}}}\psi(s) \|f\|_{\ell^p(\mathbb{Z})}\lesssim
\]
\[
 \frac{\varphi_2(2^{n^{\tau}})2^{-\chi n^{\tau}}}{|B_{2^{n^{\tau}}}|}\|f\|_{\ell^p(\mathbb{Z})}\lesssim 2^{-\chi n^{\tau}} \|f\|_{\ell^p(\mathbb{Z})}
\]
One may use a similar argument to the one presented earlier to compare $A_{2^{n^{\tau}}}$ and $D_{2^{n^{\tau}}}$. More precisely we have 
\[
\sup_{J\in\mathbb{N}_0}\sup_{I\in \mathfrak{S}_J(\mathbb{N}_0)}\|O^2_{I,J}\big(A_{2^{n^{\tau}}}f:\,n\in\mathbb{N}_0\big)\|_{\ell^p(\mathbb{Z})}\le
\]
\[
\sup_{J\in\mathbb{N}_0}\sup_{I\in \mathfrak{S}_J(\mathbb{N}_0)}\|O^2_{I,J}\big(A_{2^{n^{\tau}}}f-D_{2^{n^{\tau}}}f:\,n\in\mathbb{N}_0\big)\|_{\ell^p(\mathbb{Z})}+\sup_{J\in\mathbb{N}_0}\sup_{I\in \mathfrak{S}_J(\mathbb{N}_0)}\|O^2_{I,J}\big(D_{2^{n^{\tau}}}f:\,n\in\mathbb{N}_0\big)\|_{\ell^p(\mathbb{Z})}\lesssim
\]
\begin{equation}\label{2nts}
\Big\|\Big(\sum_{n\in\mathbb{N}_0}|A_{2^{n^{\tau}}}f-D_{2^{n^{\tau}}}f|^2\Big)^{1/2} \Big\|_{\ell^p(\mathbb{Z})}+\sup_{J\in\mathbb{N}_0}\sup_{I\in \mathfrak{S}_J(\mathbb{N}_0)}\|O^2_{I,J}\big(D_{2^{n^{\tau}}}f:\,n\in\mathbb{N}_0\big)\|_{\ell^p(\mathbb{Z})}
\end{equation}
Using the estimate $\|A_{2^{n^{\tau}}}f-D_{2^{n^{\tau}}}f\|_{\ell^p(\mathbb{Z})}\lesssim 2^{-\chi n^{\tau}} \|f\|_{\ell^p(\mathbb{Z})}$, we can bound the first term of equation ($\ref{2nts}$) by $C_{p,\tau}\|f\|_{\ell^p(\mathbb{Z})}$ and our task has been reduced to estimating the 2-oscillations of $D_{2^{n^{\tau}}}$. In fact, we will be able to estimate the 2-oscillations of $D_{n}$ by comparing it with $H_n$ by adapting the strategy of \cite{UA}, see section 5. For convenience, let $\Psi(k)=\sum_{1\le s\le k}\psi(s)$. We perform summation by parts
\[
D_kf(x)=\frac{1}{\Psi(k)}\sum_{1\le s \le k}\psi(s)f(x-s)=
\]
\[
\frac{1}{\Psi(k)}\Big(\psi(k)\sum_{1\le s\le k}f(x-s)-\sum_{1\le s\le k-1}\Big(
\sum_{1\le l \le s}f(x-l)\Big)(\psi(s+1)-\psi(s))\Big)=
\]
\[
\frac{k\psi(k)}{\Psi(k)}\sum_{1\le s\le k}\frac{f(x-s)}{k}-\sum_{1\le s\le k-1}\frac{s(\psi(s+1)-\psi(s))}{\Psi(k)}
\sum_{1\le l \le s}\frac{f(x-l)}{s}=
\]
\[
\frac{k\psi(k)}{\Psi(k)}H_kf(x)-\sum_{1\le s\le k-1}\frac{s(\psi(s+1)-\psi(s))}{\Psi(k)}H_sf(x)=\sum_{s=1}^{\infty}\lambda_{s}^kH_sf(x)
\]
where
\[ 
\lambda_{s}^k= \left\{
\begin{array}{ll}
     \frac{s(\psi(s)-\psi(s+1))}{\Psi(k)}&
     \text{if }1\le s\le k-1\\
      \frac{k\psi(k)}{\Psi(k)}&\text{if }s=k\\
      0&\text{if }s>k\\
\end{array} 
\right. 
\]
We have shown $D_kf(x)=\sum_{s=1}^{\infty}\lambda_{s}^kH_sf(x)$. Without loss of generality, $\psi(s)$ is decreasing. Thus $(\lambda_s^k)_{s,k\in\mathbb{N}}$ is a family of non-negative real numbers (in the spirit of Lemma~2 from \cite{UA}) and we note that for any $k\in\mathbb{N}$ we have that 
\[
\sum_{s=1}^\infty \lambda_s^k=\sum_{s=1}^{k-1}\frac{s(\psi(s)-\psi(s+1))}{\Psi(k)}+ \frac{k\psi(k)}{\Psi(k)}=1
\]
 and we also have that for any fixed $N\in\mathbb{N}$ the sequence $\sum_{s=1}^N\lambda_s^k$ is decreasing in $k$, since
\[
\sum_{s=1}^N\lambda_s^k=\left\{
\begin{array}{ll}
\frac{1}{\Psi(k)}\sum_{s=1}^N s(\psi(s)-\psi(s+1))&\text{if }1\le N\le k-1\\
1&\text{if }N\ge k\\
\end{array}
\right.
\]
and for any $1\le N\le k-1$ we have 
\[
\frac{1}{\Psi(k)}\sum_{s=1}^N s(\psi(s)-\psi(s+1))= \frac{1}{\Psi(k)}\Big(\sum_{s=1}^{N}\psi(s)-N\psi(N+1)\Big)\le 1
\]
We now introduce for any $k\in\mathbb{N}$ the function $N_k\colon [0,1)\to \mathbb{N}$ such that $N_k(t)=\min\{N\in\mathbb{N}:\,\sum_{i=1}^N\lambda^k_i>t\}$ and we also introduce $I_s^k=N_k^{-1}(\{s\})=\{t\in[0,1):\,N_k(t)=s\}$. We note that for all $k\in\mathbb{N}$, $N_k$ is increasing in $t$. Also, since for any fixed $N\in\mathbb{N}$ the sequence $\sum_{s=1}^N\lambda_s^k$ is decreasing in $k$, we have that for any fixed $t\in[0,1)$, the sequence $N_k(t)$ is increasing in $k$. We also note that $|I_s^k|=\lambda^k_s$. Then
\[
D_kf(x)=\sum_{s=1}^{\infty}\lambda_{s}^kH_sf(x)=\sum_{s=1}^{\infty}|I_s^k|H_sf(x)=\sum_{s=1}^{\infty}\int_{I_s^k}H_{N_k(t)}f(x)dt=\int_0^1H_{N_k(t)}f(x)dt
\]
Finally, for any $J\in\mathbb{N}_0$ and any $I=\{I_0,\dotsc,I_J\}\in \frak{S}_J(\mathbb{N}_0)$, we have that
\[
O^2_{I,J}\big(D_kf(x):\,k\in\mathbb{N}\big)=O^2_{I,J}\Big(\sum_{s=1}^{\infty}\lambda_{s}^kC_sf(x):\,k\in\mathbb{N}\Big)=O^2_{I,J}\bigg(\int_0^1C_{N_k(t)}f(x)dt:\,k\in\mathbb{N}\bigg)=
\]
\[
\bigg(\sum_{j=0}^{J-1}\sup_{I_j\le k<I_{j+1}}\bigg|\int_0^1\big(H_{N_k(t)}f(x)-H_{N_{I_j}(t)}f(x)\big)dt\bigg|^2\bigg)^{1/2}\le 
\]
\[
\bigg(\sum_{j=0}^{J-1}\bigg(\int_0^1\sup_{I_j\le k<I_{j+1}}|H_{N_k(t)}f(x)-H_{N_{I_j}(t)}f(x)|dt\bigg)^2\bigg)^{1/2}\le
\]
\[
\int_0^1\Big(\sum_{j=0}^{J-1}\sup_{I_j\le k<I_{j+1}}|H_{N_k(t)}f(x)-H_{N_{I_j}(t)}f(x)|^2\Big)^{1/2} dt
\]
and now we finish the argument by noting that
\[
\|O^2_{I,J}\big(D_kf(x):\,k\in\mathbb{N}\big)\|_{\ell^p(\mathbb{Z})}\le\bigg(\sum_{x\in\mathbb{Z}}\bigg|\int_0^1\Big(\sum_{j=0}^{J-1}\sup_{I_j\le k<I_{j+1}}|H_{N_k(t)}f(x)-H_{N_{I_j}(t)}f(x)|^2\Big)^{1/2} dt\bigg|^p\bigg)^{1/p}\le
\]
\[
\int_0^1\Big(\sum_{x\in\mathbb{Z}}\Big(\sum_{j=0}^{J-1}\sup_{I_j\le k<I_{j+1}}|H_{N_k(t)}f(x)-H_{N_{I_j}(t)}f(x)|^2\Big)^{p/2}\Big)^{1/p}dt\le
\]
\[
\int_0^1\Big(\sum_{x\in\mathbb{Z}}\Big(\sum_{j=0}^{J-1}\sup_{N_{I_j}(t)\le m<N_{I_{j+1}}(t)}|H_{m}f(x)-H_{N_{I_j}(t)}f(x)|^2\Big)^{p/2}\Big)^{1/p}dt=
\]
\[
\int_0^1\|O^2_{\{N_{I_0(t)},\dotsc,N_{I_J(t)}\},|\{N_{I_0(t)},\dotsc,N_{I_J(t)}\}|}\big(H_{m}f(x):\,m\in\mathbb{N}\big)\|_{\ell^p(\mathbb{Z})}dt\le
\]
\[
\int_0^1\sup_{\tilde{J}\in\mathbb{N}}\sup_{\tilde{I}\in\frak{S}(\mathbb{N})}\|O^2_{\tilde{I},\tilde{J}}\big(H_mf(x):\,m\in\mathbb{N}\big)\|_{\ell^p(\mathbb{Z})}dt\le C_p\|f\|_{\ell^p(\mathbb{Z})}
\]
where we have used the fact that uniform 2-oscillation $\ell^p$-estimates do hold for the standard averaging operator, see \cite{UOI} or \cite{UOIAO}, for any $p\in(1,\infty)$, or more precisely, 
\[
\sup_{\tilde{J}\in\mathbb{N}}\sup_{\tilde{I}\in\frak{S}(\mathbb{N})}\|O^2_{\tilde{I},\tilde{J}}\big(H_mf(x):\,m\in\mathbb{N}\big)\|_{\ell^p(\mathbb{Z})}\lesssim_p\|f\|_{\ell^p(\mathbb{Z})}
\]
We note that 
\[
\sup_{J\in\mathbb{N}_0}\sup_{I\in \mathfrak{S}_J(\mathbb{N}_0)}\|O^2_{I,J}\big(D_{2^{n^{\tau}}}f:\,n\in\mathbb{N}_0\big)\|_{\ell^p(\mathbb{Z})}\le\sup_{J\in\mathbb{N}}\sup_{I\in\frak{S}(\mathbb{N})}\|O^2_{I,J}\big(D_kf(x):\,k\in\mathbb{N}\big)\|_{\ell^p(\mathbb{Z})}\lesssim_p\|f\|_{\ell^p(\mathbb{Z})}
\]
This establishes the estimate ($\ref{U2OE}$).
\end{proof}
\section{Vector-Valued maximal estimates and concluding the proof of Theorem~$\ref{MPTC}$}
In this section we establish vector-valued estimates for the maximal function corresponding to $M_t$. We fix a set $B$ as in the introduction with $c_1\in[1,2)$, $c_2\in[1,6/5)$. By the Calder\'on Transference Principle, in order  to establish estimate ($\ref{VVMFE}$), it suffices to show the following.

\begin{proposition}For any $p\in(1,\infty)$, there exists a constant $C_p$ such that for any $(f_j)_{j\in\mathbb{Z}}\in \ell^p\big(\mathbb{Z};\ell^2(\mathbb{Z})\big)$ we have 
\[\bigg\|\Big(\sum_{j\in\mathbb{Z}}\big(\sup_{t\in[1,\infty)}M_t|f_j|\big)^2\Big)^{1/2}\bigg\|_{\ell^p(\mathbb{Z})}\le C_p\Big\|\Big(\sum_{j\in\mathbb{Z}}|f_j|^2\Big)^{1/2}\Big\|_{\ell^p(\mathbb{Z})}
\] 
\end{proposition}
\begin{proof}
Firstly, we note that $\sup_{t\in[1,\infty)}M_t|f|(x)\lesssim\sup_{n\in\mathbb{N}_0}M_{2^n}|f|(x)$ and thus
\[
\bigg\|\Big(\sum_{j\in\mathbb{Z}}\big(\sup_{t\in[1,\infty)}M_t|f_j|\big)^2\Big)^{1/2}\bigg\|_{\ell^p(\mathbb{Z})}\lesssim\bigg\|\Big(\sum_{j\in\mathbb{Z}}\big(\sup_{n\in\mathbb{N}_0}M_{2^n}|f_j|\big)^2\Big)^{1/2}\bigg\|_{\ell^p(\mathbb{Z})}\le
\]
\[
\bigg\|\Big(\sum_{j\in\mathbb{Z}}\big(\sup_{n\in\mathbb{N}_0}\big|M_{2^n}|f_j|-A_{2^n}|f_j|\big|+\sup_{n\in\mathbb{N}_0}A_{2^n}|f_j|\big)^2\Big)^{1/2}\bigg\|_{\ell^p(\mathbb{Z})}\le
\]
\[
\bigg\|\Big(\sum_{j\in\mathbb{Z}}\big(\sup_{n\in\mathbb{N}_0}\big|M_{2^n}|f_j|-A_{2^n}|f_j|\big|^2\big)\Big)^{1/2}\bigg\|_{\ell^p(\mathbb{Z})}+\bigg\|\Big(\sum_{j\in\mathbb{Z}}\big(\sup_{n\in\mathbb{N}_0}A_{2^n}|f_j|\big)^2\Big)^{1/2}\bigg\|_{\ell^p(\mathbb{Z})}\le
\]
\begin{equation}\label{VVSeq}
\bigg\|\Big(\sum_{j\in\mathbb{Z}}\sum_{n\in\mathbb{N}_0}\big|M_{2^n}|f_j|-A_{2^n}|f_j|\big|^2\Big)^{1/2}\bigg\|_{\ell^p(\mathbb{Z})}+\bigg\|\Big(\sum_{j\in\mathbb{Z}}\big(\sup_{n\in\mathbb{N}_0}A_{2^n}|f_j|\big)^2\Big)^{1/2}\bigg\|_{\ell^p(\mathbb{Z})}
\end{equation}
We focus on the first term. For $p=2$ we note that 
\[
\bigg\|\Big(\sum_{j\in\mathbb{Z}}\sum_{n\in\mathbb{N}_0}\big|M_{2^n}|f_j|-A_{2^n}|f_j|\big|^2\Big)^{1/2}\bigg\|_{\ell^2(\mathbb{Z})}=\Big(\sum_{x\in\mathbb{Z}}\sum_{j\in\mathbb{Z}}\sum_{n\in\mathbb{N}_0}\big|M_{2^n}|f_j|-A_{2^n}|f_j|\big|^2\Big)^{1/2}=
\]
\[
\Big(\sum_{j\in\mathbb{Z}}\sum_{n\in\mathbb{N}_0}\big\|M_{2^n}|f_j|-A_{2^n}|f_j|\big\|^2_{\ell^2(\mathbb{Z})}\Big)^{1/2}
\]
We note that Plancherel theorem combined with Lemma~$\ref{TrEst}$ yield the following
\[
\|M_{2^n}|f_j|-A_{2^n}|f_j|\|_{\ell^2(\mathbb{Z})}\lesssim 2^{-\chi n} \|f_j\|_{\ell^2(\mathbb{Z})}
\]
and thus
\[
\bigg\|\Big(\sum_{j\in\mathbb{Z}}\sum_{n\in\mathbb{N}_0}\big|M_{2^n}|f_j|-A_{2^n}|f_j|\big|^2\Big)^{1/2}\bigg\|_{\ell^2(\mathbb{Z})}\lesssim\Big(\sum_{j\in\mathbb{Z}}\sum_{n\in\mathbb{N}_0}2^{-2\chi n} \|f_j\|^2_{\ell^2(\mathbb{Z})}\Big)^{1/2}\lesssim \Big(\sum_{j\in\mathbb{Z}}\|f_j\|^2_{\ell^2(\mathbb{Z})}\Big)^{1/2}=
\]
\[
\Big(\sum_{j\in\mathbb{Z}}\sum_{x\in\mathbb{Z}}|f_j(x)|^2 \Big)^{1/2}= \Big\|\Big(\sum_{j\in\mathbb{Z}}|f_j|^2\Big)^{1/2}\Big\|_{\ell^2(\mathbb{Z})} 
\]
For the case of $p\neq 2$, we proceed in a manner identical to the one of the previous section. Firstly, let us assume that $p\in(2,\infty)$. We fix a $p_0>p$ and we note that there exists a positive constant $C$ such that
\[
\|M_{2^{n}}f-A_{2^{n}}f\|_{\ell^{p_0}(\mathbb{Z})}\le 
\|M_{2^{n}}f\|_{\ell^{p_0}(\mathbb{Z})}+\|A_{2^{n}}f\|_{\ell^{p_0}(\mathbb{Z})}\le C\|f\|_{\ell^{p_0}(\mathbb{Z})}
\] since
\[
\bigg\|\frac{1}{|B_t|}\sum_{s\in B_t}f(\cdot-s)\bigg\|_{\ell^{p_0}(\mathbb{Z})}\le \|f\|_{\ell^{p_0}(\mathbb{Z})}\text{ \& }
\bigg\|\frac{1}{|B_t|}\sum_{1\le s\le t}\psi(s)f(\cdot-s)\bigg\|_{\ell^{p_0}(\mathbb{Z})}\le \frac{1}{|B_t|}\sum_{1\le s\le t}\psi(s)\|f\|_{\ell^{p_0}(\mathbb{Z})}\lesssim \|f\|_{\ell^{p_0}}
\]
We may choose $\theta\in (0,1)$ such that $\frac{1}{p}=\frac{\theta}{2}+\frac{1-\theta}{p_0}$ and use Riesz–Thorin interpolation theorem. Since for any $n\in\mathbb{N}_0$ we have
\[
\|M_{2^{n}}f-A_{2^{n}}f\|_{\ell^{2}(\mathbb{Z})}\le C2^{-\chi n}\|f\|_{\ell^2(\mathbb{Z})}\text{ \& }\|M_{2^{n}}f-A_{2^{n}}f\|_{\ell^{p_0}(\mathbb{Z})}\le C\|f\|_{\ell^{p_0}(\mathbb{Z})}
\]
we interpolate to obtain
\begin{equation}\label{ERES}
\|M_{2^{n}}f-A_{2^{n}}f\|_{\ell^{p}(\mathbb{Z})}\le (C  2^{-\chi n})^{\theta}C^{1-\theta}\|f\|_{\ell^{p}(\mathbb{Z})}=C (2^{-\chi \theta})^{n} \|f\|_{\ell^{p}(\mathbb{Z})}
\end{equation}
Thus if we let $T_n\colon \ell^{p}(\mathbb{Z})\to \ell^p(\mathbb{Z})$ such that $T_nf=M_{2^n}f-A_{2^n}f$, then we know that $T_n$ is a bounded linear operator with $\|T_n\|_{\ell^p(\mathbb{Z})\to\ell^p(\mathbb{Z})}\le C(2^{\chi\theta})^{-n}$. Thus, we know that $T_n$ has an $\ell^2$-valued extension (see \cite{G}, page 386) with the same norm, that is
\begin{equation}\label{VVERES}
\Big\|\big(\sum_{j\in\mathbb{Z}}|T_n(f_j)|^2\big)^{1/2}\Big\|_{\ell^p(\mathbb{Z})}\le C(2^{\chi\theta})^{-n}\Big\|\big(\sum_{j\in\mathbb{Z}}|f_j|^2\big)^{1/2}\Big\|_{\ell^p(\mathbb{Z})}
\end{equation}
Finally, since $p/2>1$, we get
\[
\bigg\|\Big(\sum_{j\in\mathbb{Z}}\sum_{n\in\mathbb{N}_0}\big|M_{2^n}|f_j|-A_{2^n}|f_j|\big|^2\Big)^{1/2}\bigg\|_{\ell^p(\mathbb{Z})}=\bigg(\sum_{x\in\mathbb{Z}}\Big(\sum_{n\in\mathbb{N}_0}\sum_{j\in\mathbb{Z}}\big|M_{2^n}|f_j|(x)-A_{2^n}|f_j|(x)\big|^2\Big)^{p/2}\bigg)^{1/p}\le
\]
\[
\bigg(\sum_{n\in\mathbb{N}_0}\Big(\sum_{x\in\mathbb{Z}}\Big(\sum_{j\in\mathbb{Z}}\big|T_n|f_j|(x)\big|^2\Big)^{p/2}\Big)^{2/p}\bigg)^{1/2}\le \Big(\sum_{n\in\mathbb{N}_0}\Big\|\Big(\sum_{j\in\mathbb{Z}}\big|T_n|f_j|\big|^2\Big)^{1/2}\Big\|_{\ell^p(\mathbb{Z})}^2\Big)^{1/2}\le
\]
\[
\Big(
\sum_{n\in\mathbb{N}_0}C^2(2^{2\chi\theta})^{-n}\Big\|\big(\sum_{j\in\mathbb{Z}}|f_j|^2\big)^{1/2}\Big\|^2_{\ell^p(\mathbb{Z})}
\Big)^{1/2}\le C\sum_{n\in\mathbb{N}_0}(2^{2\chi\theta})^{-n}\Big\|\big(\sum_{j\in\mathbb{Z}}|f_j|^2\big)^{1/2}\Big\|_{\ell^p(\mathbb{Z})}\lesssim_{p} \Big\|\big(\sum_{j\in\mathbb{Z}}|f_j|^2\big)^{1/2}\Big\|_{\ell^p(\mathbb{Z})} 
\]
For $p\in (1,2)$ the situation is similar, we choose $p_0\in (1,p)$, and $\theta\in (0,1)$ such that $\frac{1}{p}=\frac{\theta}{2}+\frac{1-\theta}{p_0}$ and Riesz–Thorin interpolation theorem yields the estimate of ($\ref{ERES}$), which in turn implies the estimate ($\ref{VVERES}$). Since $p<2$, we have
\[
\bigg\|\Big(\sum_{j\in\mathbb{Z}}\sum_{n\in\mathbb{N}_0}\big|M_{2^n}|f_j|-A_{2^n}|f_j|\big|^2\Big)^{1/2}\bigg\|_{\ell^p(\mathbb{Z})}=\bigg(\sum_{x\in\mathbb{Z}}\Big(\sum_{n\in\mathbb{N}_0}\sum_{j\in\mathbb{Z}}\big|M_{2^n}|f_j|(x)-A_{2^n}|f_j|(x)\big|^2\Big)^{p/2}\bigg)^{1/p}=
\]
\[
\bigg(\sum_{x\in\mathbb{Z}}\Big(\sum_{n\in\mathbb{N}_0}\Big[\Big(\sum_{j\in\mathbb{Z}}\big|T_n|f_j|(x)\big|^2\Big)^{1/2}\Big]^2\Big)^{p/2}\bigg)^{1/p}\le\bigg(\sum_{x\in\mathbb{Z}}\sum_{n\in\mathbb{N}_0}\Big[\Big(\sum_{j\in\mathbb{Z}}\big|T_n|f_j|(x)\big|^2\Big)^{1/2}\Big]^p\bigg)^{1/p}=
\]
\[
\big(\sum_{n\in\mathbb{N}_0}\Big\|\Big(\sum_{j\in\mathbb{Z}}\big|T_n|f_j|\big|^2\big)^{1/2}\Big\|_{\ell^p(\mathbb{Z})}^p\Big)^{1/p}\lesssim  \Big( \sum_{n\in\mathbb{N}_0} 2^{-p\chi\theta n}\big\|\big(\sum_{j\in\mathbb{Z}}|f_j|^2\big)^{1/2}\big\|^p_{\ell^p(\mathbb{Z})}\Big)^{1/p}\lesssim_p
\big\|\big(\sum_{j\in\mathbb{Z}}|f_j|^2\big)^{1/2}\big\|_{\ell^p(\mathbb{Z})}
\]
We have bounded appropriately the first term of ($\ref{VVSeq}$)  and all is left is to bound the second term. We firstly observe that for any $n\in\mathbb{N}_0$ we get
\[
A_{2^n}|f|(x)=\frac{1}{|B_{2^n}|}\sum_{1\le s\le 2^n}\psi(s)|f(x-s)|\le \frac{1}{|B_{2^n}|}\sum_{k=0}^n\sum_{2^k\le s< 2^{k+1}}\psi(s)|f(x-s)|\le
\]
\[
\frac{1}{|B_{2^n}|}\sum_{k=0}^n|B\cap[2^k,2^{k+1})|\frac{1}{|B\cap[2^k,2^{k+1})|}\sum_{2^k\le s< 2^{k+1}}\psi(s)|f(x-s)|\le
\]
\[
\frac{|B_{2^{n+1}}|}{|B_{2^n}|}\sup_{k\in\mathbb{N}_0}\frac{1}{|B\cap[2^k,2^{k+1})|}\sum_{2^k\le s< 2^{k+1}}\psi(s)|f(x-s)|\lesssim
\]
\[
\frac{\varphi_2(2^{n+1})}{\varphi_2(2^n)} \sup_{k\in\mathbb{N}_0}\frac{\varphi_2'(2^k)}{|B\cap[2^k,2^{k+1})|}\sum_{2^k\le s< 2^{k+1}}|f(x-s)|\lesssim\sup_{k\in\mathbb{N}_0}\frac{\varphi_2'(2^k)2^k}{|B\cap[2^k,2^{k+1})|}\frac{1}{2^k}\sum_{2^k\le s< 2^{k+1}}|f(x-s)|\lesssim
\]
\[
\sup_{k\in\mathbb{N}_0}\frac{1}{2^k}\sum_{2^k\le s< 2^{k+1}}|f(x-s)|\lesssim \sup_{k\in\mathbb{N}_0}\frac{1}{2^k}\sum_{1\le s\le  2^{k}}|f(x-s)|=\sup_{k\in\mathbb{N}_0}H_{2^k}|f|(x)
\]
since\[\frac{\varphi_2'(2^k)2^k}{|B\cap[2^k,2^{k+1})|}\simeq \frac{\varphi_2(2^k)}{|B\cap[2^k,2^{k+1})|}\simeq 1 
\]
Since $n\in\mathbb{N}_0$ was arbitrary, we have shown that $\sup_{n\in\mathbb{N}_0}A_{2^n}|f|(x)\lesssim\sup_{n\in\mathbb{N}_0}C_{2^n}|f|(x)$, and thus the second term of $\ref{VVSeq}$ may dominated by
\[
\bigg\|\Big(\sum_{j\in\mathbb{Z}}\big(\sup_{n\in\mathbb{N}_0}A_{2^n}|f_j|\big)^2\Big)^{1/2}\bigg\|_{\ell^p(\mathbb{Z})}\lesssim \bigg\|\Big(\sum_{j\in\mathbb{Z}}\big(\sup_{n\in\mathbb{N}_0}H_{2^n}|f_j|\big)^2\Big)^{1/2}\bigg\|_{\ell^p(\mathbb{Z})}\lesssim_p\Big\|\big(\sum_{j\in\mathbb{Z}}|f_j|^2\big)^{1/2}\Big\|_{\ell^p(\mathbb{Z})}
\]
This completes the proof.
\end{proof}
The work of Section 3 and 4 proves Proposition~$\ref{OM}$.
We immediately describe how Proposition~$\ref{OM}$ implies Theorem~$\ref{MPTC}$.
\begin{proof}[Proof of Theorem~$\ref{MPTC}$ assuming Proposition~$\ref{OM}$] We simply apply Proposition~4.1 from \cite{MOE} to establish a multi-parameter uniform 2-oscillation estimate, which according to Remark~2.4 together with Proposition~2.8, page 15, from the same paper yield the desired result.
\end{proof}
\section{Proof of the weak-type (1,1) inequality}\label{WT11Inh}
	Throughout this section we have fixed a set $B$ with $\varphi_1\simeq \varphi_2$ and $c_1\in (1,30/29)$. All constants may depend on $\varphi_1,\varphi_2$ and $\psi$ but on nothing else, unless stated otherwise. We remind the reader that it suffices to establish the weak type (1,1) bound for the smooth dyadic maximal function
\[
\mathcal{M}(f)(x)=\sup_{k\in\mathbb{N}_0} \{K_{2^k}*|f|(x)\text{, where }K_N(x)=\frac{1}{\varphi_2(N)}\sum_{n\in B}\delta_{n}(x)\eta\Big(\frac{n}{N}\Big)\text{, see Remark~$\ref{SDef}$.}
\]
	The next two Lemmas are devoted to studying the properties of $K_N*\widetilde{K}_N$ and they will be key ingredients for establishing the weak-type (1,1) bound. 
	
	\begin{lemma}\label{1Aprox}
		There exists a positive constant $C$ such that for all $N\in\mathbb{N}$ and all $x\in\mathbb{Z}$ with $C\le |x|\le \varphi_1(N)$, we have that $|K_N*\tilde{K}_N(x)|\le C N^{-1}$.
	\end{lemma}
	\begin{proof}
		We have assumed that $\varphi_2\simeq \varphi_1$ and thus, by Lemma $\ref{BVBI}$, there exists a uniform bound $C_B$ for the length of intervals contained in $B$. For any $|x|\in\mathbb{Z}$ such that $x\ge C_B+1$, we get
		\[
		K_N*\widetilde{K}_N(x)=\frac{1}{\varphi_2(N)^2}\sum_{n\in\mathbb{Z}}1_{B}(n)1_{B}(n+x)\eta\Big(\frac{n}{N}\Big)\eta\Big(\frac{n+x}{N}\Big)
		\]
		and with a change of variables we see that $K_N*\widetilde{K}_N(x)=K_N*\widetilde{K}_N(-x)$, and thus, without loss of generality, let us assume that $x \ge C_B+1$. Since the $\supp(\eta)\subseteq (1/2,4)$ and $0\le\eta(x)\le 1$, we have
		\[
		K_N*\widetilde{K}_N(x)\le \frac{1}{\varphi_2(N)^2}|\{\,n\in\mathbb{Z}\,:\,n,n+x\in B\cap (N/2,4N)\,\}|
		\]
		and all is left to do is estimate the cardinality of that set. Let $\mathcal{A}_N^x=\{\,n\in\mathbb{Z}\,:\,n,n+x\in B\cap (N/2,4N)\,\}$, and notice that for any $n\in\mathcal{A}_N^x$, we have that there exists a unique $s,m\in\mathbb{N}_0$ such that
		\[
		0\le \varphi_1(n+x)-(m+s)<\psi(n+x)\text{ and }0\le \varphi_1(n)-m<\psi(n)
		\]
		Notice that since $x>C_B$, we have that $s\ge 1$, since $n$ and $n+x$ cannot correspond to the same $m$. By combining the previous set of inequalities we obtain
		\[
		\varphi_1(n+x)-\varphi_1(n)-\psi(n+x)<s<\varphi_1(n+x)-\varphi_1(n)+\psi(n)
		\]
		For constants depending only on $\varphi_1,\varphi_2$ and $\psi$, we have
		\[
		\varphi_1(n+x)-\varphi_1(n)+\psi(n)\le C x\varphi_1'(N)+C\varphi_2'(N)\le C x\varphi_1'(N)+C\varphi_1'(N)=C(x+1)\frac{\varphi_1(N)}{N}\le Cx\frac{\varphi_1(N)}{N}
		\]
		and similarly
		\[
		\varphi_1(n+x)-\varphi_1(n)-\psi(n+x)\ge cx\varphi_1'(N)-c\varphi_2'(N)\ge cx\varphi_1'(N)-c\varphi_1'(N)\ge c(x-1)\frac{\varphi_1(N)}{N}\ge cx\frac{\varphi_1(N)}{N}  
		\]
		Thus $s\simeq \frac{x\varphi_1(N)}{N}$. We have that $m\le \varphi_1(4N)\lesssim \varphi_1(N)$ and $m>\varphi_1(n)-\psi(n)\ge \varphi_1(n)-C\varphi_1'(n)\ge \varphi_1(n)\big(1-C/n)\ge C\varphi_1(N)$, when $n\in (N/2,4N)$ thus $m\simeq \varphi_1(N)$.   We also note that
		\[
		h_1(\varphi_1(n)-\psi(n))<h_1(m)\le n\text{ and } h_1(\varphi_1(n+x)-\psi(n+x))<h_1(m+s)\le n+x
		\]
		and thus
		\[
		x-\big(n+x-h_1(\varphi_1(n+x)-\psi(n+x))\big)<h_1(m+s)-h_1(m)<x+\big(n-h_1(\varphi_1(n-\psi(n)))\big)
		\]
		Note that for all $l$ we get that $l-h_1(\varphi_1(l)-\psi(l))=h_1(\varphi_1(l))-h_1(\varphi_1(l)-\psi(l))=\psi(l)h_1'(\xi_l)$, for some $\xi_l\in (\varphi_1(l)-\psi(l),\varphi_1(l))$, and thus $l-h_1(\varphi_1(l)-\psi(l))\lesssim \varphi_2'(l)h_1'(\varphi_1(l))=\frac{\varphi_2'(l)}{\varphi_1'(l)}$. Since $\varphi_2\simeq \varphi_1$, we get that there exists an absolute constant $T$, such that $l-h_1(\varphi_1(l)-\psi(l))\le T$, and thus
		\[
		h_1(m+s)-h_1(m)\in B(x,T)
		\]
		Consider the set $\mathcal{B}_N^x=\{(s,m)\in\mathbb{N}\times\mathbb{N}\,:\,s\simeq \frac{x\varphi_1(N)}{N},\,m\simeq \varphi_1(N),\,h_1(m+s)-h_1(m)\in B(x,T)\}$ and note that for any $(s,m)\in \mathcal{B}^x_N$ there are at most $C_B$ number of $n's$ in $\mathcal{A}^x_N$ corresponding to $m$. Therefore $|\mathcal{A}^x_N|\lesssim  |\mathcal{B}^x_N|$, and everything reduces to estimating $|\mathcal{B}^x_N|$. For every $s\ge 1 $ such that $s\simeq \frac{x\varphi_1(N)}{N}$, we wish to estimate the number of $m$'s such that $(s,m)\in \mathcal{B}^x_N$. Notice that if we define $g(m)=h_1(m+s)-h_1(m)$ then by the Mean Value Theorem we get
		\[
		g(m+1)-g(m)=h_1(m+1+s)-h_1(m+1)-h_1(m+s)+h_1(m)=h_1'(m+s+\xi_1)-h_1'(m+\xi_2)=
		\]
		\[
		(s+\xi_1-\xi_2)h_1''(m+\xi_3)
		\]
		for some $\xi_1,\xi_2\in(0,1)$ and $\xi_3\in (m,m+s+1)$. Thus
		\[
		g(m+1)-g(m)\simeq sh_1''(\varphi_1(N))\simeq \frac{sN}{\varphi_1(N)^2}\simeq \frac{x}{\varphi_1(N)}
		\]
		Since $x-T<h_1(m+s)-h_1(m)<x+T$, according to the previous calculation, for any fixed $s\simeq \frac{x\varphi_1(N)}{N}$, we have that there are at most $1+C \frac{T\varphi_1(N)}{x}\lesssim \frac{\varphi_1(N)}{x}$, where in the last estimate we used that $x\le\varphi_1(N)$. The number of $s$'s in $[1,C\frac{x\varphi_1(N)}{N}]$, where $C$ is the implied fixed constant appearing in $s\simeq \frac{x\varphi_1(N)}{N}$, are bounded by $C\frac{x\varphi_1(N)}{N}$, and therefore
		\[
		|\mathcal{B}^x_N|\lesssim \frac{x\varphi_1(N)}{N} \frac{\varphi_1(N)}{x}=\frac{\varphi_1(N)^2}{N} 
		\]  This implies that
		\[
		K_N*\widetilde{K}_N(x)\lesssim \frac{\varphi_1(N)^2}{N\varphi_2(N)^2}\lesssim \frac{1}{N}
		\]and the proof is complete.
	\end{proof}
	
	\begin{lemma}\label{2Aprox}
		There exists a real number $\chi>0$ such that $K_N*\widetilde{K}_N(x)=G_N(x)+E_N(x)$ for all $|x|>\varphi_1(N)$ where
		\[
		G_N(x)=\frac{1}{\varphi_2(N)^2}\sum_{n\in\mathbb{Z}}\psi(n)\psi(n+|x|)\eta\Big(\frac{n}{N}\Big)\eta\Big(\frac{n+x}{N}\Big)\text{ and } E_N(x)=K_N*\widetilde{K}_N(x)-G_N(x)
		\]
		We also have that
		$G_N(x)\lesssim N^{-1}$, $|G_N(x+h)-G_N(x)|\lesssim N^{-2}|h|$ and $|E_N(x)|\lesssim N^{-1-\chi}$.
	\end{lemma}
	\begin{proof}
		We note that for all $n\in \mathbb{N}$ we have that $1_B(n)=\lfloor \varphi_1(n) \rfloor-\lfloor \varphi_1(n)-\psi(n)\rfloor$, see \cite[Lemma~2.2]{HLMP}. We can therefore split our kernel to several manageable pieces.
		\[
		K_N*\widetilde{K}_N(x)=\frac{1}{\varphi_2(N)^2}\sum_{n\in\mathbb{Z}}1_B(n)1_B(n+x)\eta\Big(\frac{n}{N}\Big)\eta\Big(\frac{n+x}{N}\Big)=
		\]
		
		\[
		\frac{1}{\varphi_2(N)^2}\sum_{n\in\mathbb{N}}\big(\lfloor \varphi_1(n) \rfloor-\lfloor \varphi_1(n)-\psi(n)\rfloor\big)\big(\lfloor \varphi_1(n+x) \rfloor-\lfloor \varphi_1(n+x)-\psi(n+x)\rfloor\big)\eta\Big(\frac{n}{N}\Big)\eta\Big(\frac{n+x}{N}\Big)
		\]
		We will exploit a famous truncated Fourier Series. More precisely, we know that if $\Phi(x)=\{x\}-1/2$ then for all $M\in\mathbb{N}$ we get 
		\[
		\Phi(x)=\sum_{0<|m|\le M}\frac{1}{2 \pi i m}e^{-2\pi i m x}+O\bigg(\min\bigg\{1,\frac{1}{M\|x\|}\bigg\}\bigg)\text{ (see section~2 from \cite{PSP}).}
		\]
		Importantly, we also have
		\[
		\min\bigg\{1,\frac{1}{M\|x\|}\bigg\}=\sum_{m\in\mathbb{Z}}b_me^{2\pi i m x}
		\]
		where $\|x\|=\min\{|x-n|:\,n\in\mathbb{Z}\}$, and
		\[
		b_m\lesssim \min\bigg\{\frac{\log(M)}{M},\frac{1}{|m|},\frac{M}{|m|^2}\bigg\}
		\]
		Finally, we can rewrite 
		\[
		1_B(n)=\lfloor \varphi_1(n) \rfloor-\lfloor \varphi_1(n)-\psi(n)\rfloor=\varphi_1(n)-\{\varphi_1(n)\}-\big(\varphi_1(n)-\psi(n)-\{\varphi_1(n)-\psi(n)\}\big)=
		\]
		\[
		\psi(n)+\{\varphi_1(n)-\psi(n)\}-\{\varphi_1(n)\}=\psi(n)+(\{\varphi_1(n)-\psi(n)\}-1/2)-(\{\varphi_1(n)\}-1/2)=
		\]
		\[
		\psi(n)+\Phi(\varphi_1(n)-\psi(n))-\Phi(\varphi_1(n))
		\]
		Let's use the truncated Fourier Series and define
		\[
		\Delta_M(n)=\sum_{0<|m|\le M}\frac{1}{2 \pi i m}e^{-2\pi i m (\varphi_1(n)-\psi(n))}-\sum_{0<|m|\le M}\frac{1}{2 \pi i m}e^{-2\pi i m \varphi_1(n)}=\]
		\[
		\sum_{0<|m|\le M}\frac{e^{-2\pi i m\varphi_1(n)}}{2 \pi i m}\big(e^{2\pi i m \psi(n)}-1\big)
		\]
		and 
		\[
		\Pi_M(n)=\Big(\Phi(\varphi_1(n)-\psi(n))-\Phi(\varphi_1(n))\Big)-\Delta_M(n)
		\]
		and thus 
		\[
		\Pi_M(n)=O\bigg(\min\bigg\{1,\frac{1}{M\|\varphi_1(n)-\psi(n)\|}\bigg\}\bigg)+O\bigg(\min\bigg\{1,\frac{1}{M\|\varphi_1(n)\|}\bigg\}\bigg)
		\]
		Thus
		\[
		1_B(n)=\psi(n)+\Delta_M(n)+\Pi_M(n)
		\]
		Returning back to the splitting
		\[
		K_N*\widetilde{K}_N(x)=\frac{1}{\varphi_2(N)^2}\sum_{n\in\mathbb{N}}\big(\psi(n)+\Delta_M(n)+\Pi_M(n)\big)\big(\psi(n+x)+\Delta_M(n+x)+\Pi_M(n+x)\big)\eta\Big(\frac{n}{N}\Big)\eta\Big(\frac{n+x}{N}\Big)=
		\]
		\begin{align*}
			&\frac{1}{\varphi_2(N)^2}\sum_{n\in\mathbb{N}}\psi(n)\psi(n+x)\eta\Big(\frac{n}{N}\Big)\eta\Big(\frac{n+x}{N}\Big)+\\
			&\frac{1}{\varphi_2(N)^2}\sum_{n\in\mathbb{N}}\psi(n)\Delta_M(n+x)\eta\Big(\frac{n}{N}\Big)\eta\Big(\frac{n+x}{N}\Big)+\\
			&\frac{1}{\varphi_2(N)^2}\sum_{n\in\mathbb{N}}\psi(n)\Pi_M(n+x)\eta\Big(\frac{n}{N}\Big)\eta\Big(\frac{n+x}{N}\Big)+\\
			&\frac{1}{\varphi_2(N)^2}\sum_{n\in\mathbb{N}}\Delta_M(n)\psi(n+x)\eta\Big(\frac{n}{N}\Big)\eta\Big(\frac{n+x}{N}\Big)+\\
			&\frac{1}{\varphi_2(N)^2}\sum_{n\in\mathbb{N}}\Pi_M(n)\psi(n+x)\big)\eta\Big(\frac{n}{N}\Big)\eta\Big(\frac{n+x}{N}\Big)+\\
			&\frac{1}{\varphi_2(N)^2}\sum_{n\in\mathbb{N}}\Delta_M(n)\Delta_M(n+x)\eta\Big(\frac{n}{N}\Big)\eta\Big(\frac{n+x}{N}\Big)+\\
			&\frac{1}{\varphi_2(N)^2}\sum_{n\in\mathbb{N}}\Delta_M(n)\Pi_M(n+x)\eta\Big(\frac{n}{N}\Big)\eta\Big(\frac{n+x}{N}\Big)+\\
			&\frac{1}{\varphi_2(N)^2}\sum_{n\in\mathbb{N}}\Pi_M(n)\Delta_M(n+x)\eta\Big(\frac{n}{N}\Big)\eta\Big(\frac{n+x}{N}\Big)+\\
			&\frac{1}{\varphi_2(N)^2}\sum_{n\in\mathbb{N}}\Pi_M(n)\Pi_M(n+x)\eta\Big(\frac{n}{N}\Big)\eta\Big(\frac{n+x}{N}\Big)=
		\end{align*}
		$=I_1(x)+I_2(x)+\dotsm+I_9(x)$. Let $G_N(x)=I_1(x)$ and $E_N(x)=\sum_{i=2}^9I_i(x)$. Let's firstly estimate $I_1$, we have
		\[
		I_1(x)=\frac{1}{\varphi_2(N)^2}\sum_{\substack{N/2\le n\le 4N\\N/2\le n+x\le 4N}}\psi(n)\psi(n+x)\eta\Big(\frac{n}{N}\Big)\eta\Big(\frac{n+x}{N}\Big)\lesssim \frac{N\varphi_2'(N)^2}{\varphi_2(N)^2}\lesssim N^{-1}
		\]
		and for any $h\in\mathbb{Z}$ we have
		\[
		|I_1(x+h)-I_1(x)|\le\frac{1}{\varphi_2(N)^2}\sum_{n\in\mathbb{N}}\psi(n)\eta\Big(\frac{n}{N}\Big)\bigg|\psi(n+x+h)\eta\Big(\frac{n+x+h}{N}\Big)-\psi(n+x)\eta\Big(\frac{n+x}{N}\Big)\bigg|=
		\]
		\[
		\frac{1}{\varphi_2(N)^2}\sum_{n\in\mathbb{N}}\psi(n)\eta\Big(\frac{n}{N}\Big)\bigg|\int_{n+x}^{n+x+h}\bigg(\psi'(t)\eta(t/N)-\psi(t)\eta'(t/N)\frac{1}{N}\bigg)dt\bigg|
		\] 
		It suffices to consider $x,h$ such that $n+x,n+x+h\in[N/2,4N]$ since the integrand is zero outside that interval. Thus we get
		\[
		|I_1(x+h)-I_1(x)|\lesssim \frac{\varphi_2'(N)}{\varphi_2(N)^2}N|h|(\varphi_2''(N)+\varphi_2'(N)/N)\lesssim N^{-2}|h|
		\]  
		This shows the properties of $G_N$ claimed in the Lemma. Now we bound $E_N$. Let's start with $I_2$. We can rewrite $I_2$ as
		\[
		I_2(x)=\frac{1}{\varphi_2(N)^2}\sum_{n\in\mathbb{N}}\psi(n)\Delta_M(n+x)\eta\Big(\frac{n}{N}\Big)\eta\Big(\frac{n+x}{N}\Big)=
		\]
		\[
		\frac{1}{\varphi_2(N)^2}\sum_{n\in\mathbb{N}}\psi(n)\sum_{0<|m|\le M}\frac{e^{-2\pi i m\varphi_1(n+x)}}{2 \pi i m}\big(e^{2\pi i m \psi(n+x)}-1\big)\eta\Big(\frac{n}{N}\Big)\eta\Big(\frac{n+x}{N}\Big)=
		\]
		\[
		\frac{1}{\varphi_2(N)^2}\sum_{0<|m|\le M}\frac{1}{2\pi i m}\sum_{\substack{N/2\le n\le 4N\\N/2\le n+x\le 4N}}e^{2\pi i (-m)\varphi_1(n+x)}\bigg(\psi(n)\eta\Big(\frac{n}{N}\Big)\eta\Big(\frac{n+x}{N}\Big)\big(e^{2\pi i m \psi(n+x)}-1\big)\bigg)
		\]
		According to Corollary 3.12 in \cite{W11}, if we let $F_m^x(n)=\psi(n)\eta\Big(\frac{n}{N}\Big)\eta\Big(\frac{n+x}{N}\Big)\big(e^{2\pi i m \psi(n+x)}-1\big)$, we have that for all $m\in \mathbb{Z}\setminus \{0\}$
		\[
		\bigg|\sum_{\substack{N/2< n\le 4N\\N/2< n+x\le 4N}}e^{2\pi i (-m)\varphi_1(n+x)}F_m^x(n)\bigg|\lesssim
		\]
		\[
		|m|^{1/2}N(\varphi_2(N)\sigma(N))^{-1/2}\bigg(\sup_{\substack{N/2< n\le 4N\\N/2< n+x\le 4N}}|F_m^x(n)|+N\sup_{\substack{N/2< n\le 4N\\N/2< n+x\le 4N}}|F_m^x(n+1)-F_m^x(n)|\bigg)
		\] 
		Let us follow the notation of \cite{W11} and write $N_{1,x}=\max\{N/2,N/2-x\}$ and $N_{2,x}=\min\{4N,4N-x\}$. For all $n\in (N_{1,x},N_{2,x}]$ we get
		\[
		|F_m^x(n)|=\psi(n)\eta\Big(\frac{n}{N}\Big)\eta\Big(\frac{n+x}{N}\Big)\big|e^{2\pi i m \psi(n+x)}-1\big|\lesssim \frac{\varphi_2(N)}{N}|2\pi i m \psi(n+x)|\lesssim \frac{\varphi_2(N)^2}{N^2}|m|
		\]
		where we used that for all real numbers $x$ we have $|e^{ix}-1|\le |x|$. Similarly using the mean value theorem and the estimate
		\[
		\bigg|\frac{d\big(F_m^x(t)\big)}{dt}\bigg|\le\Big|\psi'(t)\eta\Big(\frac{n}{N}\Big)\eta\Big(\frac{n+x}{N}\Big)\big(e^{2\pi i m \psi(n+x)}-1\big)\Big|+\Big|\psi(t)\eta'\Big(\frac{n}{N}\Big)\frac{1}{N}\eta\Big(\frac{n+x}{N}\Big)\big(e^{2\pi i m \psi(n+x)}-1\big)\Big|+
		\]
		\[
		\Big|\psi(t)\eta\Big(\frac{n}{N}\Big)\eta'\Big(\frac{n+x}{N}\Big)\frac{1}{N}\big(e^{2\pi i m \psi(n+x)}-1\big)\Big|+\Big|\psi(t)\eta\Big(\frac{n}{N}\Big)\eta\Big(\frac{n+x}{N}\Big)\big((2\pi i m \psi'(n+x))e^{2\pi i m \psi(n+x)}\big)\Big|\lesssim
		\]
		\[
		\frac{\varphi_2(N)^2}{N^3}|m|
		\]
		Therefore we can bound $I_2$
		\[
		|I_2(x)|\lesssim \frac{1}{\varphi_2(N)^2}\sum_{0<|m|\le M}\frac{1}{|m|}|m|^{1/2}N(\varphi_2(N)\sigma(N))^{-1/2}\bigg(\frac{\varphi_2(N)^2}{N^2}|m|\bigg)
		\]
 Since $\gamma\in (29/30,1)$, we get that for $M=N^{1+2\chi+\varepsilon}\varphi_2(N)^{-1}$, $\chi=1-\gamma$ and $\varepsilon<\chi/10$ we get 
		\[
		|I_2(x)|\lesssim \frac{M^{3/2}}{N\varphi_2(N)^{1/2}\sigma(N)^{1/2}}=\frac{N^{3/2+5/2\chi+3\varepsilon}}{N^{1+\chi}\varphi_2(N)^{2}\sigma(N)^{1/2}}
		\]
		We use that for all $\varepsilon_1>0$ we have $\sigma(x)\gtrsim_{\varepsilon_1} x^{-\varepsilon_1}$ and $\varphi_2(x)\gtrsim_{\varepsilon_1} x^{\gamma -\varepsilon_1}$ to get
		\[
		|I_2(x)|\lesssim \frac{1}{N^{1+\chi}}N^{3/2+5/2\chi+3\varepsilon-2\gamma+2\varepsilon_1+\varepsilon_1/2}
		\]
		For a fixed $\varepsilon_1=\varepsilon\in(0,\chi/10)$ we get
		\[
		|I_2(x)|\lesssim N^{-1-\chi}N^{3/2-2\gamma+6\varepsilon}\lesssim N^{-1-\chi}
		\]
		since
		\[
		3/2+5/2\chi-2\gamma+6\varepsilon<0\iff 3+5\chi-4\gamma+4\varepsilon<0\iff 4(1-\gamma)+5(1-\gamma)+6/10(1-\gamma)<1 \iff
		\]
		
		\[
		96/10(1-\gamma)<1 
		\] which is true since $\gamma>29/30$.
		Therefore we have show that $|I_2(x)|\lesssim N^{-1-\chi}$, as desired. The term $I_2$ is treated similarly
		\[
		I_4(x)=\frac{1}{\varphi_2(N)^2}\sum_{n\in\mathbb{N}}\psi(n+x)\Delta_M(n)\eta\Big(\frac{n}{N}\Big)\eta\Big(\frac{n+x}{N}\Big)=
		\]
		\[
		\frac{1}{\varphi_2(N)^2}\sum_{n\in\mathbb{N}}\psi(n+x)\sum_{0<|m|\le M}\frac{e^{-2\pi i m\varphi_1(n)}}{2 \pi i m}\big(e^{2\pi i m \psi(n)}-1\big)\eta\Big(\frac{n}{N}\Big)\eta\Big(\frac{n+x}{N}\Big)=
		\]
		\[
		\frac{1}{\varphi_2(N)^2}\sum_{0<|m|\le M}\frac{1}{2\pi i m}\sum_{\substack{N/2\le n\le 4N\\N/2\le n+x\le 4N}}e^{2\pi i (-m)\varphi_1(n)}\bigg(\psi(n+x)\eta\Big(\frac{n}{N}\Big)\eta\Big(\frac{n+x}{N}\Big)\big(e^{2\pi i m \psi(n)}-1\big)\bigg)
		\]
		Using Corollary 3.12 in \cite{W11}, for $G_m^x(n)=\psi(n+x)\eta\Big(\frac{n}{N}\Big)\eta\Big(\frac{n+x}{N}\Big)\big(e^{2\pi i m \psi(n)}-1\big)$, we obtain in an almost identical fashion the bound $|I_4(x)|\lesssim N^{-1-\chi}$. We will now deal with $I_3,I_5,I_7,I_8,I_9$. We are going to follow the recipe of \cite{W11}, and we are going to use Lemma 3.18~\cite{W11}. Let's state it here.
		\begin{lemma}
			Let $N\ge 1$, $p,q\in\{0,1\}$, $x\in\mathbb{Z}$ and $M\ge 1$. Then
			\[
			\sum_{n\in\mathbb{N}} \min\bigg\{1,\frac{1}{M\|\varphi_1(n+px+q)\|}\bigg\}\eta\Big(\frac{n}{N}\Big)\eta\Big(\frac{n+x}{N}\Big)\lesssim \frac{N\log(M)}{M}+\frac{NM^{1/2}\log(M)}{(\sigma_1(N)\varphi_1(N))^{1/2}}
			\]  
		\end{lemma}
		We will use the lemma above as well as the appropriate extension of the lemma for our needs.
		
		\begin{lemma}
			Let $N\ge 1$, $p\in\{0,1\}$, $x\in\mathbb{Z}$ and $M\ge 1$. Then
			\[
			\sum_{n\in\mathbb{N}} \min\bigg\{1,\frac{1}{M\|\varphi_1(n+px)-\psi(n+px)\|}\bigg\}\eta\Big(\frac{n}{N}\Big)\eta\Big(\frac{n+x}{N}\Big)\lesssim \frac{N\log(M)}{M}+\frac{NM^{1/2}\log(M)}{(\sigma_1(N)\varphi_1(N))^{1/2}}
			\]  
		\end{lemma}
		\begin{proof}
			We have that 
			\[
			\min\bigg\{1,\frac{1}{M\|x\|}\bigg\}=\sum_{m\in\mathbb{Z}}b_me^{2\pi i m x}
			\]
			and
			\[
			b_m\lesssim \min\bigg\{\frac{\log(M)}{M},\frac{1}{|m|},\frac{M}{|m|^2}\bigg\}
			\]
			Thus 
			\[
			\sum_{n\in\mathbb{N}} \min\bigg\{1,\frac{1}{M\|\varphi_1(n+px)-\psi(n+px)\|}\bigg\}\eta\Big(\frac{n}{N}\Big)\eta\Big(\frac{n+x}{N}\Big)\lesssim
			\]
			\[
			\sum_{n+px\in(N/2,4N]} \min\bigg\{1,\frac{1}{M\|\varphi_1(n+px)-\psi(n+px)\|}\bigg\}=
			\]
			\[
			\sum_{n+px\in(N/2,4N]}\sum_{m\in\mathbb{Z}}b_me^{2\pi i m(\varphi_1(n+px)-\psi(n+px)) }\lesssim \sum_{m\in\mathbb{Z}}|b_m|\bigg|\sum_{n+px\in(N/2,4N]}e^{2\pi i m (\varphi_1(n+px)-\psi(n+px))}\bigg|\lesssim 
			\]
			
			\[
			\frac{N\log(M)}{M}+ \sum_{m\in\mathbb{Z}}|b_m||m|^{1/2}N\big(\varphi_1(N)\sigma_1(N)\big)^{-1/2}\lesssim \frac{N\log(M)}{M}+\sum_{0<|m|\le M} \frac{\log(M)}{M}|m|^{1/2}N\big(\varphi_1(N)\sigma_1(N)\big)^{-1/2}\]
			
			\[
			+\sum_{|m|> M}\frac{M}{|m|^2}|m|^{1/2}N\big(\varphi_1(N)\sigma_1(N)\big)^{-1/2}\lesssim \frac{N\log(M)}{M}+ N\log(M)M^{1/2}\big(\varphi_1(N)\sigma_1(N)\big)^{-1/2}
			\]
			where we used Lemma~4.1 from \cite{RTP} to obtain the estimate for $\big|\sum_{n+px\in(N/2,4N]}e^{2\pi i m (\varphi_1(n+px)-\psi(n+px))}\big|$.
		\end{proof}
		Using the two lemmas above, together with the trivial estimates $\psi(x),|\Delta_M(x)|,|\Pi_M(x)|\lesssim 1$ (since $1_B(n)=\psi(n)+\Delta_M(n)+\Pi_M(n)$) we may estimate
		\[
		|I_3(x)|+|I_5(x)|+|I_7(x)|+|I_8(x)|+|I_9(x)|\lesssim
		\]
		\[ \frac{1}{\varphi_2(N)^2}\sum_{n\in\mathbb{N}}\sum_{p\in\{0,1\}}\bigg(\min\bigg\{1,\frac{1}{M\|\varphi_1(n+px)-\psi(n+px)\|}\bigg\}+\min\bigg\{1,\frac{1}{M\|\varphi_1(n+px)\|}\bigg\}\bigg)\eta\Big(\frac{n}{N}\Big)\eta\Big(\frac{n+x}{N}\Big)\lesssim
		\]
		\[
		\frac{1}{\varphi_2(N)^2}\bigg(\frac{N\log(M)}{M}+\frac{NM^{1/2}\log(M)}{(\sigma_1(N)\varphi_1(N))^{1/2}}\bigg)\lesssim \frac{N\log(N)}{\varphi_2(N)N^{1+2\chi+\varepsilon}}+\frac{N^{3/2+\chi+\varepsilon/2}\log(N)}{\varphi_2(N)^3\sigma_1(N)^{1/2}}\lesssim
		\]
		\[
		N^{-1-\chi}\log(N)N^{1-\varepsilon-\gamma-\chi+\varepsilon/2}+N^{-1-\chi}\big(N^{5/2+2\chi+\varepsilon/2-3\gamma+3\varepsilon/2+\varepsilon}\log(N)\big)\lesssim N^{-1-\chi}
		\]
		since $\chi=1-\gamma$ and $5/2+2\chi-3\gamma+3\varepsilon<0\iff 5/2+2\chi+3(1-\gamma)+3\varepsilon<3\iff 5(1-\gamma)<1/2$
		$\iff 10(1-\gamma)<1$ which is true since $\gamma\in (29/30,1)$.
		Finally, for $I_6$, we get
		\[
		I_6(x)=\frac{1}{\varphi_2(N)^2}\sum_{n\in\mathbb{N}}\Delta_M(n)\Delta_M(n+x)\eta\Big(\frac{n}{N}\Big)\eta\Big(\frac{n+x}{N}\Big)=
		\]
		\[
		\frac{1}{\varphi_2(N)^2}\sum_{n\in\mathbb{N}}\sum_{0<|m_1|,|m_2|\le M}\frac{e^{-2\pi i m_1\varphi_1(n)}}{2 \pi i m_1}\big(e^{2\pi i m_1 \psi(n)}-1\big)\frac{e^{-2\pi i m_2\varphi_1(n+x)}}{2 \pi i m_2}\big(e^{2\pi i m_2 \psi(n)}-1\big)\eta\Big(\frac{n}{N}\Big)\eta\Big(\frac{n+x}{N}\Big)=
		\]
		\[
		\frac{1}{\varphi_2(N)^2}\sum_{0<|m_1|,|m_2|\le M}\frac{1}{(2\pi i)^2m_1m_2}\sum_{N_{1,x}<n\le N_{2,x}}e^{-2\pi i m_1\varphi_1(n)-2\pi im_2\varphi_1(n+x)}F^x_{m_1,m_2}(n)
		\]
		where $F^x_{m_1,m_2}(n)= \big(e^{2\pi i m_1 \psi(n)}-1\big)\big(e^{2\pi i m_2 \psi(n)}-1\big)\eta\Big(\frac{n}{N}\Big)\eta\Big(\frac{n+x}{N}\Big)$. Therefore we have that
		\[
		|I_6(x)|\lesssim \frac{1}{\varphi_2(N)^2}\sum_{0<|m_1|,|m_2|\le M} \frac{1}{|m_1m_2|}\bigg|\sum_{N_{1,x}<n\le N_{2,x}}e^{-2\pi i m_1\varphi_1(n)-2\pi im_2\varphi_1(n+x)}F^x_{m_1,m_2}(n)\bigg|
		\]
		For all $n\in(N_{1,x},N_{2,x}]$ we have $|F^x_{m_1,m_2}(n)|\lesssim |m_1m_2|\psi(N)^2\lesssim |m_1m_2|\frac{\varphi_2(N)^2}{N^2}$ and also by the mean value theorem together with the following calculation 
		\[
		\Big|\frac{d\big(F_{m_1,m_2}^x(t)\big)}{dt}\Big|\lesssim |m_1|\psi'(N)|m_2|\psi(N)+|m_1|\psi(N)|m_2|\psi'(N)+1/N|m_1m_2|\psi(N)^2\lesssim |m_1m_2|\frac{\varphi_2(N)^2}{N^3}
		\]
		we get $\sup_{N_{1,x}<n\le N_{2,x}} \{|F_{m_1,m_2}^x(n+1)-F_{m_1,m_2}^x(n)|\}\lesssim |m_1m_2|\frac{\varphi_2(N)^2}{N^3} $. We let $m=\max\{m_1,m_2\}$ and we use Corollary 3.12 from \cite{W11} for $\alpha=0$ and $\kappa=1$ to obtain
		\[
		\bigg|\sum_{N_{1,x}<n\le N_{2,x}}e^{-2\pi i m_1\varphi_1(n)-2\pi im_2\varphi_1(n+x)}F^x_{m_1,m_2}(n)\bigg|\lesssim
		\]
		\[ \max\{m_1,m_2\}^{2/3}N^{4/3}\sigma_1(N)^{-1/3}\varphi_1(N)^{-2/3}|m_1m_2|\varphi_2(N)^2N^{-2} 
		\]
		We can now finish our estimates for $I_6$
		\[
		|I_6(x)|\lesssim \sum_{0<|m_1|,|m_2|\le M}\max\{m_1,m_2\}^{2/3}N^{4/3}\sigma_1(N)^{-1/3}\varphi_1(N)^{-2/3}N^{-2}=\frac{\sum_{0<|m_1|,|m_2|\le M}\max\{m_1,m_2\}^{2/3}}{\varphi_1(N)^{2/3}N^{2/3}\sigma_1(N)^{1/3}}\lesssim
		\]
		\[
		\frac{M^{8/3}}{\varphi_1(N)^{2/3}N^{2/3}\sigma_1(N)^{1/3}}\lesssim_{\varepsilon_1} \frac{N^{8/3+16/3\chi+8/3\varepsilon}}{\varphi_1(N)^{10/3}N^{2/3}N^{-\varepsilon_1}}\lesssim_{\varepsilon_1} \frac{N^{8/3+16/3\chi+8/3\varepsilon}}{N^{10/3\gamma-\varepsilon_1}N^{2/3}N^{-\varepsilon_1}} 
		\]
		We wish to have that $8/3+16/3\chi+8/3\varepsilon-10/3\gamma+2\varepsilon_1-2/3\le -1-\chi$ but we have that
		\[
		8/3+16/3\chi+8/3\varepsilon-10/3\gamma+2\varepsilon_1-2/3\le -1-\chi\iff
		\] 
		\[
		10(1-\gamma)+19\chi+8\varepsilon+6\varepsilon_1\le 1
		\]And we can choose $\varepsilon_1>0$ to make this true. 
	\end{proof}
	
	We use Lemma~$\ref{1Aprox}$ and Lemma~$\ref{2Aprox}$ to prove the weak-type (1,1) estimates of Theorem $\ref{WT11}$. We state and prove a general Theorem that allows us to conclude. It is a natural extension of Theorem 6.1 in \cite{W11}, and the novelty lies in our handling of the problematic initial part of $K_N*\widetilde{K}_N$.
	
	\begin{theorem}\label{absThm}
		Let $\mathcal{M}f(x)=\sup_{n\in\mathbb{N}}|K_n*f(x)|$ be the maximal function corresponding to a family of nonnegative kernels $\big(K_n\big)_{n\in\mathbb{N}}\subseteq \ell^1(\mathbb{Z})$ such that $||\mathcal{M}f||_{\ell^{\infty}(\mathbb{Z})}\lesssim ||f||_{\ell^{\infty}(\mathbb{Z})}$ for all $f\in \ell^{\infty}(\mathbb{Z})$ and let $\big(F_n\big)_{n\in\mathbb{Z}}$ be a family of nonnegative functions. Assume that there are sequences $(d_n)_{n\in\mathbb{N}}$, $(D_n)_{n\in\mathbb{N}}\subseteq [1,\infty)$ such that $|supp(K_n)|=d_n$, $supp(K_n)\subseteq[0,D_n]$, $supp(F_n)\subseteq [-D_n,D_n]$, $d_n\le D_n^{\varepsilon_0}$ for some $\varepsilon_0\in (0,1)$ and assume there exists a finite constant $M>1$ such that $Md_n\le d_{n+1}$ and $MD_n\le D_{n+1}\le 2^{n+1}$ for all $n\in\mathbb{N}$. Also, assume that exists a real number $\varepsilon_1>0$ such that for all $n\in\mathbb{N}$ and $x\in\mathbb{Z}$ we have
		\[
		|K_n*\widetilde{K}_n(x)-F_n(x)|\lesssim D_n^{-1-\varepsilon_1}
		\]
		and assume that there exists a constant $A>0$ such that
		\begin{equation}\label{eq1}
			F_n(x)\lesssim d_n^{-1} \text{ for all }x\text{ with }|x|\le A\text{ and } |F_n(x)|\lesssim D_n^{-1}\text{ for all }x\text { with }|x|>A
		\end{equation}
		Finally, assume that there exists an $\varepsilon_2\in (0,1]$ such that 
		\begin{equation}\label{eq2}
			|F_n(x+y)-F_n(x)|\lesssim  D_n^{-2}|y|\text{ whenever }|x|,|x+y|\gtrsim d_n^{\varepsilon_2}
		\end{equation}
		Then we have that there exists a constant $C>0$ such that
		\[
		||\mathcal{M}f||_{\ell^{1,\infty}(\mathbb{Z})}\le C ||f||_{\ell^1(\mathbb{Z})}\text{ for all }f\in\ell^1(\mathbb{Z})
		\]
	\end{theorem}
Before proving the Theorem let us briefly show how it implies the weak-type (1,1) bound. 
\begin{proof}[Proof of Theorem $\ref{WT11}$]

By letting $K_n(x)=\frac{1}{\varphi_2(2^n)}\sum_{k\in B}\delta_k(x)\eta\Big(\frac{k}{2^n}\Big)$, $d_n\simeq\varphi_1(2^n)$, $D_n\simeq2^n$ we can apply the theorem for 
	\[
	F_n(x)=\left\{\begin{array}{ll}
		(K_n*\widetilde{K}_n)(x),& 0\le |x| \le \varphi_1(2^n) \\
		G_{2^n}(x), &|x|>\varphi_1(2^n)\\
	\end{array} 
	\right.
	\]
	Lemma~$\ref{2Aprox}$ guarantees the existence of a real number $\varepsilon_1>0$ such that $|K_n*\widetilde{K}_n(x)-F_n(x)|\lesssim D_n^{-1-\varepsilon_1}$ for $|x|> \varphi_1(2^n)$ and for smaller values of $x$ the estimate is trivially established from the definition of $F_n$. We also have
	\[
	K_n*\widetilde{K}_n(x)=\frac{1}{\varphi_2(2^n)^2}\sum_{k\in\mathbb{N}}1_B(k)1_B(k+x)\eta\Big(\frac{k}{2^n}\Big)\eta\Big(\frac{k+x}{2^n}\Big)
	\]
	and thus if $C$ is the constant appearing in Lemma~$\ref{1Aprox}$, for all integers $x$ such that $|x|\le C$, we get that
	\[
	|F_n(x)|=|K_n*\widetilde{K}_n(x)|\lesssim \frac{1}{\varphi_1(2^n)}\simeq d_n^{-1}
	\]
	and for all $x$ such that $C\le|x|\le \varphi_1(2^n)$ we get 
	\[
	|F_n(x)|=|G_{2^n}(x)|\lesssim 2^{-n}\simeq D_n^{-1}
	\]
	We can conclude by letting $\varepsilon_2=1$, and using the estimates from Lemma~$\ref{1Aprox}$ and Lemma~$\ref{2Aprox}$. This completes the proof.
	\end{proof}
	\begin{proof}[Proof of Theorem $\ref{absThm}$]
	Let $f\in\ell^1(\mathbb{Z})$ such that $f\ge 0$ and let $\alpha>0$. We will perform a subtle variation of the Calder\'on-Zygmund decomposition. There exists a family of disjoint dyadic cubes $(Q_{s,j})_{(s,j)\in\mathcal{B}}$, where $\mathcal{B}\subseteq \mathbb{N}_0\times \mathbb{Z}$ and $Q_{s,j}=[j2^s,(j+1)2^s)\cap\mathbb{Z}$ and functions $g,b$ such that 
	\begin{itemize}
		\item $f=g+b$
		\item $\|g\|_{\ell^1(\mathbb{Z})}\le \|f\|_{\ell^1(\mathbb{Z})}\text{ and }\|g\|_{\ell^{\infty}(\mathbb{Z})}\le 2\alpha$
		\item $b=\sum_{(s,j)\in\mathcal{B}}b_{s,j}\text{ where }b_{s,j}	\text{ is supported on }Q_{s,j}$
		\item $\sum_{x\in Q_{s,j}}b_{s,j}(x)=0$
		\item $\|b_{s,j}\|_{\ell^1(\mathbb{Z})}\le 4\alpha |Q_{s,j}|$
		\item $\sum_{(s,j)\in \mathcal{B}}|Q_{s,j}|\le \alpha^{-1}\|f\|_{\ell^1(\mathbb{Z})}$
	\end{itemize}
	For every $s\ge 0$ we let 
	\[
	b_s=\sum_{\substack{j\in \mathbb{Z}:\\(s,j)\in\mathcal{B}}}b_{s,j}
	\]
	and for every $n\in\mathbb{N}_0$ we decompose further
	\begin{itemize}
		\item $b_s^n(x)=b_s(x)1_{\{y\in\mathbb{Z}:\,|b_s(y)|>\alpha d_n\}}(x)$
		\item $h_s^n(x)=b_s^n(x)-b_s^n(x)=b_s(x)1_{\{y\in\mathbb{Z}:\,|b_s(y)|\le\alpha d_n\}}(x)$
		\item \[g_s^n(x)=\sum_{\substack{j\in \mathbb{Z}:\\(s,j)\in\mathcal{B}}} [h^n_s]_{Q_{s,j}}1_{Q_{s,j}}
		\]
		
		\item
		\[
		B_s^n(x)=h_s^n(x)-g_s^n(x)=\sum_{\substack{j\in \mathbb{Z}:\\(s,j)\in\mathcal{B}}} \big(h^n_s-[h^n_s]_{Q_{s,j}}\big)1_{Q_{s,j}}
		\]	 
	\end{itemize}
Let $s(n)=\min\{s\in\mathbb{N}_0\,:\,2^s\ge D_n\}$ and decompose $f$ as $g+\sum_{s\ge 0}b_s=g+\sum_{s\ge 0}(b_s^n+g_s^n+B_s^n)=\bigg(g+\sum_{s\ge 0}g^n_{s}\bigg)+\sum_{s\ge 0}b^n_s+\sum_{s=0 }^{s(n)-1}B^n_s+\sum_{s=s(n) }^{\infty}B^n_s$.
	We have that
\[
\Big|\Big\{x\in\mathbb{Z}\,:\,\sup_{n\in\mathbb{N}}|K_n*f(x)|>C\alpha\Big\}\Big|\le
\]
\[
\Big|\Big\{x\in\mathbb{Z}\,:\,\sup_{n\in\mathbb{N}}|K_n*\Big(g+\sum_{s\ge 0}g^n_s\Big)(x)|>C\alpha/4\Big\}\Big|+\Big|\Big\{x\in\mathbb{Z}\,:\,\sup_{n\in\mathbb{N}}|K_n*\Big(\sum_{s\ge 0}b^n_s\Big)(x)|>C\alpha/4\Big\}\Big|+
\]
\[
\Big|\Big\{x\in\mathbb{Z}\,:\,\sup_{n\in\mathbb{N}}|K_n*\Big(\sum_{s=s(n) }^{\infty}B^n_s\Big)(x)|>C\alpha/4\Big\}\Big|+\Big|\Big\{x\in\mathbb{Z}\,:\,\sup_{n\in\mathbb{N}}|K_n*\Big(\sum_{s=0 }^{s(n)-1}B^n_s\Big)(x)|>C\alpha/4\Big\}\Big|
\]
and our treatment will be different for each summand. The following subsections are devoted to this task, and the most difficult part will be to bound the final one, where we will exploit the cancellation properties of $B^n_s$ together with the properties of $F_n$.
\subsection{Estimates for the first three summands}

For the good part we will use $\ell^{\infty}$-bounds together with the fact that $\|\mathcal{M}\|_{\ell^{\infty}\to\ell^{\infty}}=T<\infty$. Note that
\[
\big|\sum_{s\ge 0}g^n_s(x)\big|\le \sum_{(s,j)\in\mathcal{B}}|[h^n_s]_{Q_{s,j}}1_{Q_{s,j}}(x)|\le  \sum_{(s,j)\in\mathcal{B}}[|h^n_s|]_{Q_{s,j}}1_{Q_{s,j}}(x)\le \sum_{(s,j)\in\mathcal{B}}[|b_{s,j}|]_{Q_{s,j}}1_{Q_{s,j}}(x) \le
\] 
\[
\sum_{(s,j)\in\mathcal{B}}\|b_{s,j}\|_{\ell^1(\mathbb{Z})}|Q_{s,j}|^{-1}1_{Q_{s,j}}(x)\le \sum_{(s,j)\in\mathcal{B}}4\alpha 1_{Q_{s,j}}(x)\le 4\alpha
\]
and $g(x)\le 2\alpha$ for all $x\in\mathbb{Z}$, and thus $\|g+\sum_{s\ge 0}g^n_s\|_{\ell^{\infty}(\mathbb{Z})}\le 6\alpha$ and thus $\big|K_n*\big(g+\sum_{s\ge 0}g^n_s\big)(x)\big|\le T6\alpha$ for all $x\in\mathbb{Z}$ and $n\in\mathbb{N}$. Thus for any $C>24T$ we get
\[
\Big|\Big\{x\in\mathbb{Z}\,:\,\sup_{n\in\mathbb{N}}|K_n*\Big(g+\sum_{s\ge 0}g^n_s\Big)(x)|>C\alpha/4\Big\}\Big|=0
\]
since it is the empty set.

For the second summand we use the lacunary nature of $(d_n)_{n\in\mathbb{N}}$ as well as the bounds for the cardinality of the support of $K_n$. Specifically, we have
\[
\Big|\Big\{x\in\mathbb{Z}\,:\,\sup_{n\in\mathbb{N}}|K_n*\Big(\sum_{s\ge 0}b^n_s\Big)(x)|>C\alpha/4\Big\}\Big|\le \Big|\bigcup_{n\in\mathbb{N}}\bigcup_{s\in\mathbb{N}_0}\supp{K_n*|b^n_s|}\Big|\le \sum_{n\in\mathbb{N}}\sum_{s\in\mathbb{N}_0}|\supp K_n*|b_s^n|\le
\]
\[
\sum_{n\in\mathbb{N}}\sum_{s\in\mathbb{N}_0}|\supp K_n|\cdot| \supp b_s^n|\le \sum_{n\in\mathbb{N}}\sum_{s\in\mathbb{N}_0}d_n| \{x\in\mathbb{Z}\,:\,|b_s(x)|>\alpha d_n\}|
=
\]
\[
\sum_{n\in\mathbb{N}}\sum_{s\in\mathbb{N}_0}d_n\sum_{k\ge n}| \{x\in\mathbb{Z}\,:\,\alpha d_{k+1}\ge|b_s(x)|>\alpha d_k\}|=
\]
\[
\sum_{s\in\mathbb{N}_0}\sum_{k\in\mathbb{N}}\Big(
\sum_{n=1}^k d_n\Big)| \{x\in\mathbb{Z}\,:\,\alpha d_{k+1}\ge|b_s(x)|>\alpha d_k\}|\lesssim_M
\]
\[
\alpha^{-1} \sum_{s\in\mathbb{N}_0}\sum_{k\in\mathbb{N}}\alpha d_k | \{x\in\mathbb{Z}\,:\,\alpha d_{k+1}\ge|b_s(x)|>\alpha d_k\}|\le \alpha^{-1}\sum_{s\in\mathbb{N}_0}\|b_s\|_{\ell^1(\mathbb{Z})} \lesssim \alpha^{-1}\|f\|_{\ell^1(\mathbb{Z})}\]

For the third summand we simply use the fact for any $n\in\mathbb{N}$ and any $s\ge s(n)$ we get that $2^s\ge D_n$ and thus 
\[
\supp(K_n*B^n_s)\subseteq\supp(K_n)+\supp(B^n_s)\subseteq [0,D_n]+\bigcup_{\substack{j\in\mathbb{Z}\\(s,j)\in\mathcal{B}}} Q_{s,j} \subseteq [0,2^s]+\bigcup_{\substack{j\in\mathbb{Z}\\(s,j)\in\mathcal{B}}} Q_{s,j}\subseteq \bigcup_{\substack{j\in\mathbb{Z}\\(s,j)\in\mathcal{B}}} 3Q_{s,j}
\]
where $3Q$ denotes the interval with the same center as $Q$ and three times its radius. Therefore
\[
\Big|\Big\{x\in\mathbb{Z}\,:\,\sup_{n\in\mathbb{N}}|K_n*\Big(\sum_{s=s(n) }^{\infty}B^n_s\Big)(x)|>C\alpha/4\Big\}\Big|=\Big|\bigcup_{n\in\mathbb{N}}\bigcup_{s\ge s(n)}\supp(K_n*B^n_s)\Big|\le \Big|\bigcup_{n\in\mathbb{N}}\bigcup_{s\ge s(n)}\bigcup_{\substack{j\in\mathbb{Z}\\(s,j)\in\mathcal{B}}} 3Q_{s,j}\Big|\le 
\]
\[
\Big|\bigcup_{(s,j)\in\mathcal{B}} 3Q_{s,j}\Big|\lesssim \alpha^{-1}\|f\|_{\ell^1(\mathbb{Z})}
\]

\subsection{Estimates for the fourth summand}\label{RBADPART}The fourth summand is the most difficult to estimate and here we will use the regularity of $K_n*\widetilde{K}_n$. We have
\[
\Big|\Big\{x\in\mathbb{Z}\,:\,\sup_{n\in\mathbb{N}}|K_n*\Big(\sum_{s=0 }^{s(n)-1}B^n_s\Big)(x)|>C\alpha/4\Big\}\Big|\lesssim \alpha^{-2}\sum_{x\in\mathbb{Z}}\sup_{n\in\mathbb{N}}|K_n*\Big(\sum_{s=0 }^{s(n)-1}B^n_s\Big)(x)|^2\le
\]
\[
\alpha^{-2}\sum_{x\in\mathbb{Z}}\sum_{n\in\mathbb{N}}|K_n*\Big(\sum_{s=0 }^{s(n)-1}B^n_s\Big)(x)|^2=\alpha^{-2}\sum_{n\in\mathbb{N}}\Big\|\sum_{s=0 }^{s(n)-1}K_n*B^n_s\Big\|_{\ell^2(\mathbb{Z})}^2=
\]
\[
\alpha^{-2}\sum_{n\in\mathbb{N}}\bigg(\sum_{s=0 }^{s(n)-1}\|K_n*B^n_s\|_{\ell^2(\mathbb{Z})}^2+2\sum_{0\le s_1<s_2\le s(n)-1}\langle K_n*B^n_{s_1},K_n*B^n_{s_2} \rangle_{\ell^2(\mathbb{Z})}\bigg)
\]
We need the following result to conclude.

\begin{claim}\label{CC}There exists $0<\lambda<1$ such that for all  $n\in\mathbb{N}$ and $0\le s_1\le s_2\le s(n)-1$ we get
\begin{equation}\label{difcompact}
\big|\langle K_n*B^n_{s_1},K_n*B^n_{s_2} \rangle_{\ell^2(\mathbb{Z})}\big|\lesssim \lambda^{s(n)-s_1}\alpha \|B_{s_2}^n\|_{\ell^1(\mathbb{Z})}+d_n^{-1}\sum_{|j|\le A}|\langle \delta_j* B_{s_1}^n,B_{s_2}^n\rangle_{\ell^2({\mathbb{Z})}}|
\end{equation}
\end{claim}
Assuming that ($\ref{difcompact}$) holds, let us see how we can deduce the desired estimate. We have 
\[
\alpha^{-2}\sum_{n\in\mathbb{N}}\bigg(\sum_{s=0 }^{s(n)-1}\|K_n*B^n_s\|_{\ell^2(\mathbb{Z})}^2+2\sum_{0\le s_1<s_2\le s(n)-1}\langle K_n*B^n_{s_1},K_n*B^n_{s_2} \rangle_{\ell^2(\mathbb{Z})}\bigg)\lesssim
\]
\[
\alpha^{-2}\sum_{n\in\mathbb{N}}\sum_{s=0 }^{s(n)-1}\Big(\lambda^{s(n)-s}\alpha \|B_{s}^n\|_{\ell^1(\mathbb{Z})}+d_n^{-1}\sum_{|j|\le A}|\langle \delta_j* B_{s}^n,B_{s}^n\rangle_{\ell^2({\mathbb{Z})}}|\Big)+
\]
\[
\alpha^{-2}\sum_{n\in\mathbb{N}}\sum_{0\le s_1<s_2\le s(n)-1}\Big(\lambda^{s(n)-s_1}\alpha \|B_{s_2}^n\|_{\ell^1(\mathbb{Z})}+d_n^{-1}\sum_{|j|\le A}|\langle \delta_j* B_{s_1}^n,B_{s_2}^n\rangle_{\ell^2({\mathbb{Z})}}|\Big)\lesssim
\]
\[
\alpha^{-1}\sum_{n\in\mathbb{N}}\sum_{s=0 }^{s(n)-1}\lambda^{s(n)-s}\|B_{s}^n\|_{\ell^1(\mathbb{Z})}+\alpha^{-2}\sum_{n\in\mathbb{N}}\sum_{s=0}^{s(n)-1}d_n^{-1}\sum_{|j|\le A}|\langle \delta_j* B_{s}^n,B_{s}^n\rangle_{\ell^2({\mathbb{Z})}}|+
\]
\[
\alpha^{-1}\sum_{n\in\mathbb{N}}\sum_{0\le s_1<s_2\le s(n)-1}\lambda^{s(n)-s_1} \|B_{s_2}^n\|_{\ell^1(\mathbb{Z})}+\alpha^{-2}\sum_{n\in\mathbb{N}}\sum_{0\le s_1<s_2\le s(n)-1}d_n^{-1}\sum_{|j|\le A}|\langle \delta_j* B_{s_1}^n,B_{s_2}^n\rangle_{\ell^2({\mathbb{Z})}}|
\]

For the first and the third term note that
\[
\alpha^{-1}\sum_{n\in\mathbb{N}}\sum_{0\le s_1<s_2\le s(n)-1}\lambda^{s(n)-s_1} \|B_{s_2}^n\|_{\ell^1(\mathbb{Z})}=\alpha^{-1}\sum_{n\in\mathbb{N}}\sum_{1\le  s_2\le s(n)-1}\|B_{s_2}^n\|_{\ell^1(\mathbb{Z})}\big(\sum_{0\le s_1\le s_2-1}\lambda^{s(n)-s_1}\big) \lesssim
\]
\[
\alpha^{-1}\sum_{n\in\mathbb{N}}\sum_{1\le  s_2\le s(n)-1}\lambda^{s(n)-s_2}\|B_{s_2}^n\|_{\ell^1(\mathbb{Z})} \le \alpha^{-1}\sum_{n\in\mathbb{N}}\sum_{s=0}^{s(n)-1}\lambda^{s(n)-s}\|B_{s}^n\|_{\ell^1(\mathbb{Z})}\lesssim_{\lambda}
\]
\[
 \alpha^{-1}\sum_{n\in\mathbb{N}}\sum_{s=0}^{s(n)}\lambda^{s}\|B_{s(n)-s-1}^n\|_{\ell^1(\mathbb{Z})}\lesssim \alpha^{-1}\sum_{s\in\mathbb{N}_0}\lambda^s\sum_{(t,j)\in\mathcal{B}}\|b_{t,j}\|_{\ell^1(\mathbb{Z})}\lesssim\alpha^{-1}\sum_{(t,j)\in\mathcal{B}}4\alpha |Q_{t,j}|\lesssim \alpha^{-1}\|f\|_{\ell^1(\mathbb{Z})} 
\]
where we have used the estimates from the Calder\'on-Zygmund decomposition in the beginning of the proof.

The fourth term is bounded as follows.
\[
\alpha^{-2}\sum_{n\in\mathbb{N}}\sum_{0\le s_1<s_2\le s(n)-1}d_n^{-1}\sum_{1\le |j|\le A}|\langle \delta_j*B^n_{s_1},B^n_{s_2}\rangle_{\ell^2(\mathbb{Z})}|\le
\]
\[
\alpha^{-2}\sum_{1\le |j|\le A}\sum_{n\in\mathbb{N}}\sum_{0\le s_1<s_2\le s(n)-1}d_n^{-1}\sum_{x\in\mathbb{Z}} |B^n_{s_1}(x-j)||B^n_{s_2}(x)|=
\]
\[
\alpha^{-2}\sum_{x\in\mathbb{Z}}\sum_{1\le |j|\le A}\sum_{n\in\mathbb{N}}\sum_{0\le s_1<s_2\le s(n)-1}d_n^{-1} |B^n_{s_1}(x-j)||B^n_{s_2}(x)|
\]
Fix $|j|\le A$ and $x\in\mathbb{Z}$ such that $x-j\in \supp(b_{s_0})$ for some integer $s_0$. Since the supports of $b_{s}$'s are disjoint we have that there can be at most one integer $s_0'$ such that $x\in\supp(b_{s_0'})$. Note also that $[|h^n_s|]_{Q_{s,j}}\lesssim \alpha$, as we observed earlier. We have

\begin{equation}\label{redhere}
\sum_{n\in\mathbb{N}}\sum_{0\le s_1<s_2\le s(n)-1}d_n^{-1} |B^n_{s_1}(x-j)||B^n_{s_2}(x)|\le \sum_{n\in\mathbb{N}}d_n^{-1}|B^n_{s_0}(x)||B^n_{s_0'}(x-j)|\lesssim
\end{equation}
\[
\sum_{n\in\mathbb{N}}d_n^{-1} \big(|b_{s_0}(x)|1_{\{y\in\mathbb{Z}:\,|b_{s_0}(y)|\le\alpha d_n\}}(x)+\alpha1_{\supp b_{s_0}}(x)\big)\big(|b_{s_0'}(x-j)|1_{\{y\in\mathbb{Z}:\,|b_{s_0'}(y)|\le\alpha d_n\}}(x-j)+\alpha1_{\supp b_{s_0'}}(x-j)\big)\le
\]
\[
\sum_{\substack{n\in\mathbb{N}:\\d_n\ge |b_{s_0}(x)|/\alpha\\d_n\ge|b_{s_0'}(x-j)|/\alpha}}d_n^{-1} |b_{s_0}(x)||b_{s_0'}(x-j)|+\alpha^2\sum_{n\in\mathbb{N}}d_n^{-1}1_{\supp b_{s_{0}}}(x)+
\]
\[
\sum_{\substack{n\in\mathbb{N}:\\d_n\ge |b_{s_0}(x)|/\alpha}}d_n^{-1}|b_{s_0}(x)|\alpha 1_{\supp b_{s_0'}}(x-j)+\sum_{\substack{n\in\mathbb{N}:\\d_n\ge |b_{s_0'}(x-j)|/\alpha}}d_n^{-1}|b_{s_0'}(x-j)|\alpha 1_{\supp b_{s_0}}(x)\lesssim
\]
\[
\max\{|b_{s_0}(x)|^2,|b_{s_0'}(x-j)|^2\}\min\{\alpha/|b_{s_0}(x)|,\alpha/|b_{s_0'}(x-j)|\}+\alpha^21_{\supp b_{s_0}}(x)\le
\]
\[
\sum_{s\in \mathbb{N}_0}\alpha|b_s(x)|+\alpha|b_s(x-j)|+\sum_{s\in \mathbb{N}_0}\alpha^21_{\supp b_s}(x)
\]
where we have used the existence of a finite constant $M>1$ such that $Md_n\le d_{n+1}$. We get that 
\[
\alpha^{-2}\sum_{n\in\mathbb{N}}\sum_{0\le s_1<s_2\le s(n)-1}d_n^{-1}\sum_{1\le |j|\le A}|\langle \delta_j*B^n_{s_1},B^n_{s_2}\rangle_{\ell^2(\mathbb{Z})}|\lesssim 
\]

\[
\alpha^{-2}\sum_{x\in\mathbb{Z}}\sum_{1\le |j|\le A}\Big(\sum_{s\in \mathbb{N}_0}\alpha|b_s(x)|+\alpha|b_s(x-j)|+\sum_{s\in \mathbb{N}_0}\alpha^21_{\supp b_s}(x)\Big)\lesssim
\]
\[
\alpha^{-1}\sum_{1\le |j|\le A}\sum_{s\in\mathbb{N}_0}2\|b_s\|_{\ell^1(\mathbb{Z})}+A\sum_{s\in\mathbb{N}_0}|\supp b_s|\lesssim \alpha^{-1}A\|b\|_{\ell^1(\mathbb{Z})}+A\Big|\bigcup_{(s,j)\in\mathcal{B}}Q_{s,j}\Big|\lesssim\alpha^{-1}\|f\|_{\ell^1(\mathbb{Z})} 
\]
A similar argument may be used to bound the second term. For the sake of completeness we note that 
\[
\alpha^{-2}\sum_{n\in\mathbb{N}}\sum_{s=0}^{s(n)-1}d_n^{-1}\sum_{1\le |j|\le A}|\langle \delta_j*B^n_{s},B^n_{s}\rangle_{\ell^2(\mathbb{Z})}|\le
\]
\[
\alpha^{-2}\sum_{1\le |j|\le A}\sum_{n\in\mathbb{N}}\sum_{s=0}^{s(n)-1}d_n^{-1}\sum_{x\in\mathbb{Z}} |B^n_{s}(x-j)||B^n_{s}(x)|=
\]
\[
\alpha^{-2}\sum_{x\in\mathbb{Z}}\sum_{1\le |j|\le A}\sum_{n\in\mathbb{N}}\sum_{s=0}^{s(n)-1}d_n^{-1} |B^n_{s}(x-j)||B^n_{s}(x)|
\]
Fix $|j|\le A$ and $x\in\mathbb{Z}$ such that $x-j\in \supp(b_{s_0})$ for some integer $s_0$. Since the supports of $b_{s}$'s are disjoint we have that there can be at most one integer $s_0'$ such that $x\in\supp(b_{s_0'})$. We have
\[
\sum_{n\in\mathbb{N}}\sum_{s=0}^{s(n)-1}d_n^{-1} |B^n_{s}(x-j)||B^n_{s}(x)|\le \sum_{n\in\mathbb{N}}d_n^{-1}|B^n_{s_0}(x)||B^n_{s_0'}(x-j)|
\]
In fact if $s_0'\neq s_0$, then the left-hand side equals $0$. Nevertheless, the inequality above holds for some $s_0,s_0'$ that depend on $x,j$ and by comparing it with ($\ref{redhere}$), we see that an argument identical to the one used previously may be used here. The proof will be completed once we established the estimate ($\ref{difcompact}$) of the claim. We do this in the following subsection. 
\subsection{Proof of the estimate $(\ref{difcompact})$}
Let $n\in\mathbb{N}$ and $0\le  s_1\le s_2\le s(n)-1$ and let us note that on the one hand
\[
\big|\langle K_n*B^n_{s_1},K_n*B^n_{s_2} \rangle_{\ell^2(\mathbb{Z})}\big|=\big|\langle K_n*\widetilde{K}_n*B^n_{s_1},B^n_{s_2} \rangle_{\ell^2(\mathbb{Z})}\big|\le
\]
\[
\big|\langle F_n*B^n_{s_1},B^n_{s_2} \rangle_{\ell^2(\mathbb{Z})}\big|+\big|\langle (K_n*\widetilde{K}_n-F_n)*B^n_{s_1},B^n_{s_2} \rangle_{\ell^2(\mathbb{Z})}\big|\le
\]
Now decompose $F_n(x)=F_n(x)1_{|x|\le A}+F_n(x)1_{|x|> A}=\sum_{-A \le j\le A}F_n(j)\delta_j(x)+F_n(x)1_{|x|> A}$ and let $G_n(x)=F_n(x)1_{|x|>A}$. We obtain
\[
\big|\langle K_n*B^n_{s_1},K_n*B^n_{s_2} \rangle_{\ell^2(\mathbb{Z})}\big|\le \sum_{-A\le j\le A}|F_n(j)|\big|\langle \delta_j*B^n_{s_1},B^n_{s_2} \rangle_{\ell^2(\mathbb{Z})}\big|+\big|\langle G_n*B^n_{s_1},B^n_{s_2} \rangle_{\ell^2(\mathbb{Z})}\big|+
\]
\[
+\big|\langle (K_n*\widetilde{K}_n-F_n)*B^n_{s_1},B^n_{s_2} \rangle_{\ell^2(\mathbb{Z})}\big|\lesssim d_n^{-1}\sum_{|j|\le A}\big|\langle \delta_j*B^n_{s_1},B^n_{s_2} \rangle_{\ell^2(\mathbb{Z})}\big|+\big|\langle G_n*B^n_{s_1},B^n_{s_2} \rangle_{\ell^2(\mathbb{Z})}\big|+
\]
\[
+\big|\langle (K_n*\widetilde{K}_n-F_n)*B^n_{s_1},B^n_{s_2} \rangle_{\ell^2(\mathbb{Z})}\big|
\]
In the right hand side of our inequality one of the terms  of the desired estimate already appeared and thus we can now focus on the other two summands. Let $Z_{m,n}=\big[m2^{s(n)},(m+1)2^{s(n)}\big)\cap \mathbb{Z}$, $\widetilde{Z}_{m,n}=\big[(m-1)2^{s(n)},(m+2)2^{s(n)}\big)\cap \mathbb{Z}$ and $E_n=K_n*\widetilde{K}_n-F_n$. Note that $\supp(E_n)\subseteq [-D_n,D_n]\subseteq [-2^{s(n)},2^{s(n)}]$. We have that
\[
\big|\langle (K_n*\widetilde{K}_n-F_n)*B^n_{s_1},B^n_{s_2} \rangle_{\ell^2(\mathbb{Z})}\big|=\big|\sum_{x\in\mathbb{Z}}(E_n*B^n_{s_1})(x)B^n_{s_2}(x)\big|=\big|\sum_{x\in\mathbb{Z}}\Big(\sum_{y\in\mathbb{Z}}E_n(y)B^n_{s_1}(x-y)\Big)B^n_{s_2}(x)\big|\le
\]
\[
\sum_{x\in\mathbb{Z}}\sum_{y\in\mathbb{Z}}|E_n(y)B^n_{s_1}(x-y)B^n_{s_2}(x)|\le \sum_{y\in\mathbb{Z}}\sum_{m\in\mathbb{Z}}\sum_{x\in Z_{m,n}}1_{\widetilde{Z}_{m,n}}(x-y)|E_n(y)B_{s_1}^n(x-y)B^n_{s_2}(x)|\le
\]
\[ D_n^{-1-\varepsilon_1}\sum_{m\in\mathbb{Z}}\sum_{x\in Z_{m,n}}|B^n_{s_2}(x)|\bigg(\sum_{y\in\mathbb{Z}}1_{\widetilde{Z}_{m,n}}(x-y)|B^n_{s_1}(x-y)|\bigg)\le
\]
\[
D_n^{-1-\varepsilon_1}\|B^n_{s_2}\|_{\ell^1(\mathbb{Z})}\sup_{m\in\mathbb{Z}}\|B^n_{s_1}1_{\widetilde{Z}_{m,n}}\|_{\ell^1(\mathbb{Z})}
\]
Now we note that for any $m\in\mathbb{Z}$ we have
\begin{equation}\label{prcalc}
\|B^n_{s_1}1_{\widetilde{Z}_{m,n}}\|_{\ell^1(\mathbb{Z})}=\sum_{\substack{k\in\mathbb{Z}\\(s_1,k)\in\mathcal{B}}}\sum_{x\in\mathbb{Z}}|B^n_{s_1}1_{\widetilde{Z}_{m,n}}(x)1_{Q_{s_1,k}}(x)|\le
\\
\sum_{\substack{k\in\mathbb{Z}:\,(s_1,k)\in\mathcal{B}\\Q_{s_1,k}\cap \widetilde{Z}_{m,n}\neq \emptyset}}\|B^n_{s_1}1_{Q_{s_1,k}}\|_{\ell^1(\mathbb{Z})}
\end{equation}
On the one hand
\begin{equation}\label{endest}
\|B^n_{s_1}1_{Q_{s_1,k}}\|_{\ell^1(\mathbb{Z})}\le \|h_{s_1}^n1_{Q_{s_1,k}}\|_{\ell^1(\mathbb{Z})}+\|[h_{s_1}^n]_{Q_{s_1,k}}1_{Q_{s_1,k}}\|_{\ell^1(\mathbb{Z})}\le 
2\|b_{s_1}1_{Q_{s_1,k}}\|_{\ell^1(\mathbb{Z})}\le 4\alpha|Q_{s_1,k}|\le 4\alpha 2^{s_1}
\end{equation}
and on the other hand $|\{k\in\mathbb{Z}:\,(s_1,k)\in\mathcal{B}\,\&\,Q_{s_1,k}\cap \widetilde{Z}_{m,n}\neq \emptyset\}|\lesssim 2^{s(n)-s_1}$ since $\widetilde{Z}_{m,n}$ can be partitioned into $3$ dyadic intervals, $Q_{s_1,k}$ is a dyadic interval and $s_1\le s(n)$. Therefore
\begin{equation}\label{prcalc2}
\|B^n_{s_1}1_{\widetilde{Z}_{m,n}}\|_{\ell^1(\mathbb{Z})}\lesssim 2^{s(n)-s_1}4\alpha  2^{s_1}=8\alpha 2^{s(n)-1}\le 8\alpha D_n
\end{equation}
and finally
\[
\big|\langle (K_n*\widetilde{K}_n-F_n)*B^n_{s_1},B^n_{s_2} \rangle_{\ell^2(\mathbb{Z})}\big|\lesssim D_n^{-\varepsilon_1}8\alpha \|B_{s_2}^n\|_{\ell^1(\mathbb{Z})}\lesssim 2^{-\varepsilon_1(s(n)-s_1)}\alpha \|B_{s_2}^n\|_{\ell^1(\mathbb{Z})} 
\]
since $D_n\ge 2^{s(n)-1}\ge 2^{s(n)-s_1-1}$, we get
\[
\big|\langle (K_n*\widetilde{K}_n-F_n)*B^n_{s_1},B^n_{s_2} \rangle_{\ell^2(\mathbb{Z})}\big|\lesssim \big(2^{-\varepsilon_1}\big)^{s(n)-s_1}\alpha \|B_{s_2}^n\|_{\ell^1(\mathbb{Z})}
\]
as desired.

Now we focus on the last term $\big|\langle G_n*B^n_{s_1},B^n_{s_2} \rangle_{\ell^2(\mathbb{Z})}\big|$. Note that $\supp(G_n)\subseteq [-2^{s(n)},2^{s(n)}]$. We have
\[
\big|\langle G_n*B^n_{s_1},B^n_{s_2} \rangle_{\ell^2(\mathbb{Z})}\big|=\big|\sum_{x\in\mathbb{Z}}(G_n*B^n_{s_1})(x)B^n_{s_2}(x)\big|=\big|\sum_{x\in\mathbb{Z}}\Big(\sum_{y\in\mathbb{Z}}G_n(y)B^n_{s_1}(x-y)\Big)B^n_{s_2}(x)\big|\le
\]

\[
\big|\sum_{m\in\mathbb{Z}}\sum_{x\in Z_{m,n}}\Big(\sum_{y\in\mathbb{Z}}G_n(y)B^n_{s_1}(x-y)1_{\widetilde{Z}_{m,n}}(x-y)\Big)B^n_{s_2}(x)\big|\le
\]

\[
\sum_{m\in\mathbb{Z}}\sup_{x \in Z_{m,n}}\big|G_n*(B^n_{s_1}1_{\widetilde{Z}_{m,n}})(x)\big|\sum_{x\in Z_{m,n}}|B^n_{s_2}(x)|\le \sup_{m\in\mathbb{Z}}\sup_{x \in Z_{m,n}}\big|G_n*\big(B^n_{s_1}1_{\widetilde{Z}_{m,n}}\big)(x)\big|\|B^n_{s_2}(x)\|_{\ell^1(\mathbb{Z})}
\]
Let us define $B^n_{s,j}=B^n_{s}1_{Q_{s,j}}$, and note that for any $m\in\mathbb{Z}$, $x\in Z_{m,n}$, we have
\[
\big|G_n*\big(B^n_{s_1}1_{\widetilde{Z}_{m,n}}\big)(x)\big|\le \sum_{\substack{j\in \mathbb{Z}:\\(s_1,j)\in\mathcal{B}}}\big|G_n*\big(B^n_{s_1,j}1_{\widetilde{Z}_{m,n}}\big)(x)\big|\text{, since }\sum_{\substack{j\in \mathbb{Z}:\\(s_1,j)\in\mathcal{B}}}B^n_{s,j}=B^n_{s}\text{.}
\]
We also note that $\sum_{x\in\mathbb{Z}}B^n_{s_1,j}(x)1_{\widetilde{Z}_{m,n}}(x)=0$. To see this note that if $\supp(B^n_{s_1,j})\cap \widetilde{Z}_{m,n}=\emptyset$, then it is trivial, and if they intersect, we must have that  $\supp(B^n_{s_1,j})\subseteq \widetilde{Z}_{m,n}$, since $\supp(B^n_{s_1,j})\subseteq Q_{s_1,j}$ which is a dyadic interval of length $2^{s_1}$ and  $\widetilde{Z}_{m,n}$ is the union of three dyadic intervals of larger length. In the second case we get
\[
\sum_{x\in\mathbb{Z}}B^n_{s_1,j}(x)1_{\widetilde{Z}_{m,n}}(x)=\sum_{x\in\mathbb{Z}}B^n_{s_1,j}(x)=0
\]
from the definition of $B^n_{s_1,j}(x)$. Fix $m\in\mathbb{Z}$ and $j\in \mathbb{Z}$ such that $(s_1,j)\in\mathcal{B}$ and let $x_{s_1,j}$
be the center of the cube $Q_{s_{1},j}$. Assume $x\in Z_{m,n}$ is such that $|x-x_{s_1,j}|\ge C d_n^{\varepsilon_2}+C2^{s_1}$. Using the cancellation property we have established together with the regularity assumptions for $F_n$, we get
\[
\big|G_n*\big(B^n_{s_1,j}1_{\widetilde{Z}_{m,n}}\big)(x)\big|=\Big|\sum_{y\in\mathbb{Z}}\big(G_n(x-y)-G_n(x-x_{s_1,j})\big)B^n_{s_1,j}(y)1_{\widetilde{Z}_{m,n}}(y)\Big|\lesssim
\]
\[
\sum_{y\in\mathbb{Z}}D_n^{-2}|y-x_{s_1,j}||B^n_{s_1,j}(y)1_{\widetilde{Z}_{m,n}}(y)|\lesssim D_n^{-2}2^{s_1}\|B^n_{s_1,j}1_{\widetilde{Z}_{m,n}}\|_{\ell^1(\mathbb{Z})}
\]
Here we have used the fact that $|x-y|\ge |x-x_{s_1,j}|-|x_{s_1,j}-y|\ge C d_n^{\varepsilon_2}+C2^{s_1}-2^{s_1}\gtrsim d_n^{\varepsilon_2}$ and thus we may use ($\ref{eq2}$). Taking into account ($\ref{eq1}$) and the definition of $G_n$, we get that for any $x\in\mathbb{Z}$
\[
\big|G_n*\big(B^n_{s_1,j}1_{\widetilde{Z}_{m,n}}\big)(x)\big|\lesssim D_n^{-1}\|B^n_{s_1,j}1_{\widetilde{Z}_{m,n}}\|_{\ell^1(\mathbb{Z})}
\]
Now we may estimate as follows
\[
\sup_{m\in\mathbb{Z}}\sup_{x\in Z_{m,n}}\big|G_n*\big(B^n_{s_1}1_{\widetilde{Z}_{m,n}}\big)(x)\big|\le\sup_{m\in\mathbb{Z}}\sup_{x\in Z_{m,n}}\sum_{\substack{j\in \mathbb{Z}:\\(s_1,j)\in\mathcal{B}}}\big|G_n*\big(B^n_{s_1,j}1_{\widetilde{Z}_{m,n}}\big)(x)\big|\le
\]

\[
\sup_{m\in\mathbb{Z}}\sup_{x\in Z_{m,n}}\sum_{\substack{j\in \mathbb{Z}:\,(s_1,j)\in\mathcal{B}\\|x-x_{s_1,j}|\ge C d_n^{\varepsilon_2}+C2^{s_1}}}D_n^{-2}2^{s_1}\|B^n_{s_1,j}1_{\widetilde{Z}_{m,n}}\|_{\ell^1(\mathbb{Z})}+
\sup_{m\in\mathbb{Z}}\sup_{x\in Z_{m,n}}\sum_{\substack{j\in \mathbb{Z}:\,(s_1,j)\in\mathcal{B}\\|x-x_{s_1,j}|< C d_n^{\varepsilon_2}+C2^{s_1}}}D_n^{-1}\|B^n_{s_1,j}1_{\widetilde{Z}_{m,n}}\|_{\ell^1(\mathbb{Z})}
\]
For the first summand note that 
\[
\sup_{m\in\mathbb{Z}}\sup_{x\in Z_{m,n}}\sum_{\substack{j\in \mathbb{Z}:\,(s_1,j)\in\mathcal{B}\\|x-x_{s_1,j}|\ge C d_n^{\varepsilon_2}+C2^{s_1}}}D_n^{-2}2^{s_1}\|B^n_{s_1,j}1_{\widetilde{Z}_{m,n}}\|_{\ell^1(\mathbb{Z})}\le D_n^{-2}2^{s_1}\sup_{m\in\mathbb{Z}}\sum_{\substack{j\in \mathbb{Z}:\,(s_1,j)\in\mathcal{B}}}\|B^n_{s_1,j}1_{\widetilde{Z}_{m,n}}\|_{\ell^1(\mathbb{Z})}\le
\]
\[
D_n^{-2}2^{s_1}\sup_{m\in\mathbb{Z}}\|B^n_{s_1}1_{\widetilde{Z}_{m,n}}\|_{\ell^1(\mathbb{Z})}
\]
The calculations from ($\ref{prcalc}$), ($\ref{prcalc2}$) show that $\sup_{m\in\mathbb{Z}}\|B^n_{s_1}1_{\widetilde{Z}_{m,n}}\|_{\ell^1(\mathbb{Z})}\lesssim \alpha D_n$ and thus the first summand is bounded by a constant multiple of    
\[
\alpha D_n^{-1}2^{s_1}\le \alpha 2^{1-s(n)}2^{s_1}\lesssim \alpha (1/2)^{s(n)-s_1}
\]

For the second summand we consider two cases. In the first case we assume that $2^{s_1}\le d_n^{\varepsilon_2}$. In that case, any interval of radius $\lesssim d_n^{\varepsilon_2}$ contains at most $\lesssim 2^{-s_1}d_n^{\varepsilon_2}$ sets of the form $Q_{s_1,j}$. Thus we have 
\[
\sup_{m\in\mathbb{Z}}\sup_{x\in Z_{m,n}}\sum_{\substack{j\in \mathbb{Z}:\,(s_1,j)\in\mathcal{B}\\|x-x_{s_1,j}|< C d_n^{\varepsilon_2}+C2^{s_1}}}D_n^{-1}\|B^n_{s_1,j}1_{\widetilde{Z}_{m,n}}\|_{\ell^1(\mathbb{Z})}\lesssim
\]
\[
\sup_{m\in\mathbb{Z}}\sup_{x\in Z_{m,n}}|\{j\in \mathbb{Z}:\,(s_1,j)\in\mathcal{B}\text{ and }|x-x_{s_1,j}|< 2C d_n^{\varepsilon_2}\}|D_n^{-1}\alpha 2^{s_1}\lesssim
\]
\[
2^{-s_1}d_n^{\varepsilon_2}D_n^{-1} \alpha 2^{s_1}\lesssim  \frac{\alpha d_n}{D_n^{\varepsilon_0}D_n^{1-\varepsilon_0}}\lesssim  \frac{\alpha}{(2^{s(n)})^{1-\varepsilon_0}}\lesssim \alpha (1/2^{1-\varepsilon_0})^{s(n)-s_1}
\]
where we used the estimate ($\ref{endest}$), the fact that $d_n\le D_n^{\varepsilon_0}$ and the fact that $2^{s(n)-1}<D_n\le 2^{s(n)}$. We have established the appropriate bound for the first case.

In the second case, we assume that  $2^{s_1}> d_n^{\varepsilon_2}$. In that case, any interval of radius $\lesssim 2^{s_1}$ contains at most $\lesssim 2^{s_1}2^{-s_1}=1$ sets of the form $Q_{s_1,j}$. Thus we have
\[
\sup_{m\in\mathbb{Z}}\sup_{x\in Z_{m,n}}\sum_{\substack{j\in \mathbb{Z}:\,(s_1,j)\in\mathcal{B}\\|x-x_{s_1,j}|< C d_n^{\varepsilon_2}+C2^{s_1}}}D_n^{-1}\|B^n_{s_1,j}1_{\widetilde{Z}_{m,n}}\|_{\ell^1(\mathbb{Z})}\lesssim
\]
\[
\sup_{m\in\mathbb{Z}}\sup_{x\in Z_{m,n}}|\{j\in \mathbb{Z}:\,(s_1,j)\in\mathcal{B}\text{ and }|x-x_{s_1,j}|< 2C 2^{s_1}\}|D_n^{-1}\alpha 2^{s_1}\lesssim
 D_n^{-1}\alpha 2^{s_1}\lesssim \alpha  (1/2)^{s(n)-s_1}
 \]
This concludes the second case. 

Combining everything we get
\[
\big|\langle G_n*B^n_{s_1},B^n_{s_2} \rangle_{\ell^2(\mathbb{Z})}\big|\lesssim \sup_{m\in\mathbb{Z}}\sup_{x \in Z_{m,n}}\big|G_n*\big(B^n_{s_1}1_{\widetilde{Z}_{m,n}}\big)(x)\big|\|B^n_{s_2}(x)\|_{\ell^1(\mathbb{Z})}\lesssim \big(2^{\varepsilon_0-1}\big)^{s(n)-s_1}\|B^n_{s_2}(x)\|_{\ell^1(\mathbb{Z})}
\]
since $2^{-1}\lesssim2^{\varepsilon_0-1}$. For $\lambda=\min\{2^{-\varepsilon_1},2^{\varepsilon_0-1}\}\in (0,1)$, we obtain the estimate $(\ref{difcompact})$. The proof of Theorem~$\ref{absThm}$ is complete.
\end{proof}


\begin{thebibliography}{99}
		
		\bibitem{RWC}
		J.M. Rosenblatt, M. Wierdl, 
		\emph{Pointwise ergodic theorems via harmonic analysis.} Ergodic Theory and Harmonic Analysis: Proceedings of the 1993 Alexandria Conference (K.E. 					Petersen and I. Salama, eds.), Cambridge University Press, 1995, pp. 3–152.
		
		\bibitem{BUL1}
		 Z. Buczolich,
		\emph{Universally $L^1$ good sequences with gaps tending to infinity.} Acta Math. Hungar. 117 (2007), no. 1–2, 91–14.
		
		\bibitem{ntoc}
		R. Urban, J. Zienkiewicz, 
		\emph{Weak Type (1, 1) Estimates for a Class of Discrete Rough Maximal Functions.} Math. Res. Lett. 14 (2007), no. 2, 227–237.
		
		\bibitem{W11}
		M. Mirek,
		\emph{Weak type (1, 1) inequalities for discrete rough maximal functions.} J. Anal. Mat. 127 (2015), 303–337.
		
		\bibitem{HLMP}
		B. Krause, M. Mirek, B. Trojan,
		\emph{On the Hardy--Littlewood majorant problem for arithmetic sets.} J. Funct. Anal., 271(1): 164–181, 2016.
		
		\bibitem{RTP}
		L. Daskalakis,
		\emph{Roth's theorem and the Hardy--Littlewood majorant problem for thin subsets of Primes.} Available at arXiv: \url{https://arxiv.org/abs/2212.14513}
		
		\bibitem{FE}
		C. Fefferman,
		\emph{Inequalities for strongly singular convolution operators.} Acta Math. 124 (1970), 9-36.
		
		\bibitem{MC}
		M. Christ,
		\emph{Weak type (1, 1) bounds for rough operators.} Ann. of Math. (2) 128 (1988), no. 1, 19-42.
  
		\bibitem{bg2}
	       J. Bourgain,
            \emph{On the maximal ergodic theorem for certain subsets of the integers.} Israel J. Math. 61 (1988), pp. 39-72.
 
        \bibitem{bg3}
            J. Bourgain,
            \emph{On the pointwise ergodic theorem on $L^p$ for arithmetic sets.} Israel J. Math. 61 (1988), pp. 73-84.
    
		\bibitem{BET}
		J. Bourgain, 
		\emph{Pointwise ergodic theorems for arithmetic sets, with an appendix by the author, H. Furstenberg, Y. Katznelson, and D. S. Ornstein.} Inst. Hautes Etudes Sci. Publ. Math. 69 (1989), 5–45.
  \bibitem{DF}
    N. Dunford,
    \emph{An individual ergodic theorem for non-commutative transformations.} Acta Sci. Math. Szeged 14 (1951), pp. 1-4.

\bibitem{ZG}
    A. Zygmund,
    \emph{An individual ergodic theorem for non-commutative transformations.} Acta Sci. Math. Szeged 14 (1951), pp. 103–110.
		
		\bibitem{MOE}
		M. Mirek, T. Z. Szarek, J. Wright,
		\emph{Oscillation inequalities in ergodic theory and analysis: one-parameter and multi-parameter perspectives.} Rev. Mat. Iberoam. 38 (2022), no. 7, 2249–2284.
		\bibitem{MPVBF}
    J. Bourgain, M. Mirek, E. M. Stein, J. Wright,
    \emph{On a multi-parameter variant of the Bellow-Furstenberg problem.} Available at arXiv: \url{https://arxiv.org/abs/2209.07358}
		\bibitem{SF}
		C. Fefferman, E. M. Stein,
		\emph{Some Maximal Inequalities.} American Journal of Mathematics, Vol. 93, No. 1 (Jan., 1971), pp. 107-115.
		
		\bibitem{MVV}
		M. Mirek, E. M. Stein, B. Trojan,
		\emph{$\ell^p(\mathbb{Z}^d)$-estimates for discrete operators of Radon type: Maximal functions and vector-valued estimates.}
		
		\bibitem{JI}
		M. Mirek, E. M. Stein, P. Zorin-Kranich,
		\emph{Jump inequalities for translation-invariant operators of Radon type on $\mathbb{Z}^d$.} Advances in Mathematics 365 (2020), 107065, pp. 57.
		
		\bibitem{PZ}
		P. Zorin-Kranich,
		\emph{Variation estimates for averages along primes and polynomials.} J. Funct. Anal. 268.1 (2015), pp. 210–238.
		
		\bibitem{UA}
		M. Mirek, B. Trojan, P. Zorin-Kranich,
		\emph{Variational estimates for averages and truncated singular integrals along the prime numbers.} Transactions of the American Mathematical Society 369, (2017), no. 8, 5403-5423.
		
		\bibitem{UOI}
		R.L. Jones, R. Kaufman, J.M. Rosenblatt, M. Wierdl,
		\emph{Oscillation in ergodic theory.} Ergodic Theory Dynam. Systems 18 (1998), no. 4, 889–935.
		
		\bibitem{UOIAO}
		M. Mirek, W. Słomian, T. Z. Szarek,
		\emph{Some remarks on oscillation inequalities.} Ergodic Theory Dynam. Systems, 1–30. doi:10.1017/etds.2022.77.
		
		\bibitem{JEFM}
		R.L. Jones, A. Seeger, J. Wright,
		\emph{Strong variational and jump inequalities in harmonic analysis.} Trans. Amer. Math. Soc. 360 (2008), no. 12, 6711–6742.
		
		\bibitem{G}
		L. Grafakos,
		\emph{Classical Fourier Analysis.} Vol. 249 of Graduate Texts in Mathematics, third edition, Springer.
		
		\bibitem{PSP}
		D. R. Heath–Brown,
		\emph{The Pjateckii--Sapiro prime number theorem.} J. Number Theory, 16 (1983), 242–266.






		
\end{thebibliography}
\end{document}